\documentclass[12pt]{amsart}
\usepackage{amsmath,amssymb,amsthm}
\usepackage{amsfonts, verbatim}
\usepackage{mathrsfs}
\usepackage{gensymb}
\usepackage[dvips]{graphicx}
\usepackage{color}              
\usepackage{epsfig}
\usepackage{enumerate}
\usepackage{pb-diagram,amscd,amsfonts,amsmath}
\usepackage[active]{srcltx}
\usepackage[all]{xy}
\usepackage{graphicx}

\newtheorem{theorem}{Theorem}[section]

\newtheorem{corollary}[theorem]{Corollary}
\newtheorem{lemma}[theorem]{Lemma}

\theoremstyle{definition}
\newtheorem{definition}[theorem]{Definition}

\theoremstyle{remark}

\newcommand{\ol}{\overline}

\def\scp{\mathcal{SCP}}
\def\M{\mathcal{M}}
\def\H{\mathcal{H}}
\def\C{\mathbb{C}}
\def\Dbar{\overline{\mathbb{D}}}
\def\Cbar{\overline{\mathbb{C}}}
\def\D{\mathbb{D}}
\def\sm{\setminus}
\def\lam{\mathcal{L}}
\def\uc{\mathbb{S}^1}

\def\gcd{\mathrm{GCD}}
\def\mo{\mathrm{mod}}
\def\co{\ol{\mathrm{co}}}
\def\Q{\mathbb{Q}}
\def\pr{\mathfrak{pr}}
\newcommand{\R}{\mathbb{R}}
\newcommand{\ch}{\mathrm{CH}}
\newcommand{\thl}{\mathrm{TH}}
\newcommand{\V}{\mathcal{V}}
\newcommand{\cdisk}{\overline{\mathbb{D}}}
\newcommand{\sh}{\mathrm{SH}}
\newcommand{\bc}{\overline{c}}

\newcommand{\hb}{\mathbf{h}}

\newcommand{\al}{\alpha}
\newcommand{\be}{\beta}

\newcommand{\ga}{\gamma}
\newcommand{\ta}{\theta}

\newcommand{\tro}{\tilde\rho}
\newcommand{\si}{\sigma}

\def\Ac{\mathcal{A}}                \def\Cc{\mathcal{C}}
                
\def\Hc{\mathcal{H}}

\def\Rc{\mathcal{R}}

\def\Cc{\mathcal{C}}
\def\Uc{\mathcal{U}}       \def\Wc{\mathcal{W}}

\def\Tc{\mathcal{T}}
\def\Oc{\mathcal{O}}
\def\Z{\mathbb{Z}}
\def\0{\varnothing}

\def\Rf{\mathfrak{R}}

\title{Symmetric cubic polynomials}

\begin{document}

\date{May 10, 2023}

\author[A.~Blokh]{Alexander Blokh}

\author[L.~Oversteegen]{Lex Oversteegen}

\author[N.~Selinger]{Nikita Selinger}

\author[V.~Timorin]{Vladlen Timorin}

\author[S.~Vejandla]{Sandeep Chowdary Vejandla}

\address[Alexander~Blokh, Lex~Oversteegen, Nikita Selinger, and Sandeep Chowdary Vejandla]
{Department of Mathematics\\ University of Alabama at Birmingham\\
Birmingham, AL 35294}

\address[Vladlen~Timorin]
{Faculty of Mathematics\\
National Research University Higher School of Economics\\
6 Usacheva str., Moscow, Russia, 119048}

\email[Alexander~Blokh]{ablokh@math.uab.edu}
\email[Lex~Oversteegen]{overstee@uab.edu}
\email[Nikita~Selinger]{selinger@uab.edu}
\email[Sandeep-Vejandla]{vsc4u@uab.edu}
\email[Vladlen~Timorin]{vtimorin@hse.ru}

\thanks{The second named author was partially
supported by NSF grant DMS--1807558.
The work of the fourth named author was supported by
 the Russian Science Foundation under grant no. 22-11-00177.}



\begin{abstract}
We describe a model $\M_3^{comb}$ for the boundary of the connectedness locus $\M^{sy}_3$ of the parameter space
of cubic symmetric polynomials $p_c(z)=z^3-3c^2z$. We show that there exists a monotone continuous function $\pi:\partial \M_c^{sy}\to \M_3^{comb}$
which is a homeomorphism if $\M^{sy}_3$ is locally connected.
\end{abstract}

\maketitle

\section{Introduction}

A central problem of Complex Dynamics is to describe parameter spaces of holomorphic maps. Investigating the deceptively simple
\emph{quadratic family} $f_c(z)=z^2+c$ led to an explosion of activity in the field. Aided by
computer graphics capabilities, mathematicians were able to
 make and justify many interesting discoveries concerning
the connected\-ness locus of the quadratic family, the celebrated \emph{Man\-del\-brot set} $\M_2$.

One of the first such results by Douady and Hubbard \cite{DH} was 
 that 
$\M_2$ is connected.
Later results on the combinatorial description of the structure of the Mandelbrot set were largely carried out in the language
of 
\emph{laminations} introduced by Thurston \cite{Th} (see Section \ref{s:lami} for precise definitions
and other details).
Douady constructed a topological \emph{pinched disk} model 
 of $\M_2$; Thurston made this model more explicit and described it in terms of laminations.
If $\M_2$ is locally connected (which is still an open conjecture), then it is homeomorphic to its model.
The local connectivity of the Mandelbrot set is one of the most important long standing conjectures in the field; if true, it will imply 
a more profound density of hyperbolicity 
property of the quadratic family.

Here is a quick reminder of what laminations are, and how they produce pinched disk models.
For any compact, connected, locally connected, and full subset $K$ of the complex plane $\C$,
a Riemann map $\psi \colon \Cbar \sm \Dbar \to \Cbar \sm K$ can be used to obtain a
family of external rays (the images of rays $\arg(z)=\text{const}$) and equipotentials (the images
of circles $|z|=\text{const}$); here $\D$ is the open unit disk and $\Cbar$ is the complex sphere. Since $K$ is assumed to be locally
connected, every external ray lands on a point of $K$ by Caratheodory's theorem, and the landing point
depends continuously on the ray's angle. We can define an equivalence relation by declaring two external
rays (or arguments thereof) equivalent if they land on the same point.
The edges 
(\emph{leaves})
of convex hulls of equivalence classes (finite \emph{gaps} or stand alone leaves) in the closed unit disk form a \emph{lamination} $\lam_K$.
It is easy to see that $K$ can be reconstructed by starting with $\lam_K$ and collapsing convex hulls of classes of the just
defined equivalence relation to points (in particular, one has to collapse every finite gap or stand alone leaf to a point).
Informally this process can be called ``pinching the unit disk''.


\begin{figure}
   \centering
        \includegraphics[height=6.2cm]{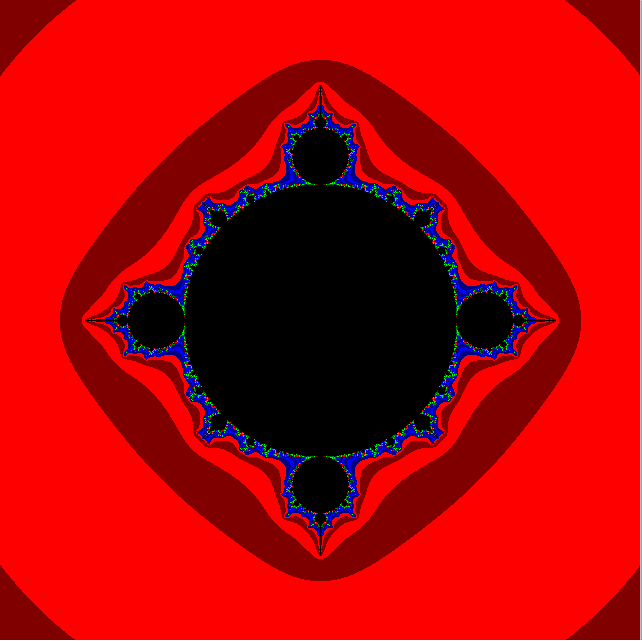}
        \includegraphics[height=6.2cm]{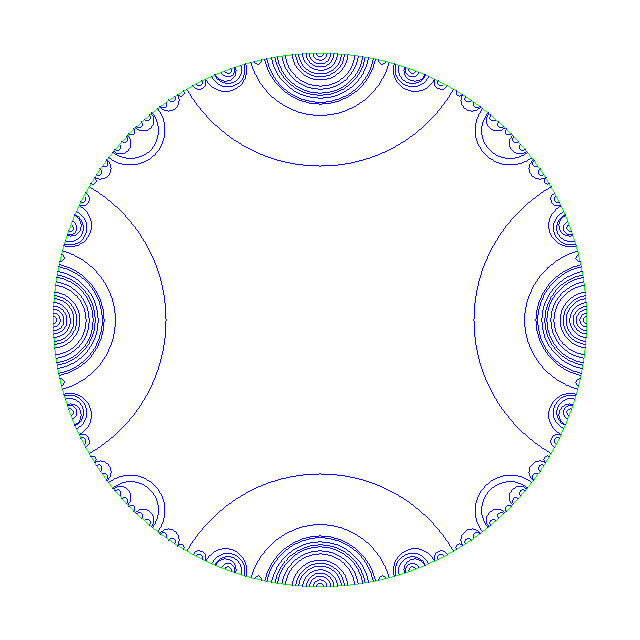}
     \caption{The parameter space of symmetric cubic  polynomials $\mathcal{SCP}$ on the left and the Symmetric Cubic Comajor Lamination $C_sCL$ on the right.}
\end{figure}

Now 
 assume that
$K$ 
is a connected and locally connected Julia set of a monic polynomial.
We can then choose a Riemann map $\psi$ as above so that $\psi(\infty)=\infty$ and $\psi'(\infty)=1$.
Everything said above applies; 
 additionally, the obtained lamination is 
 invariant with respect to the angle $d$-tupling map (multiplication of the angles by $d$),
 where $d$ is the degree of the polynomial.
Indeed, the map $\psi$ conjugates the given polynomial to $z \mapsto z^d$ on $\Cbar \sm K$.

Thus, the study of parameter spaces of polynomials is deeply related to the study of the spaces of possible laminations
of these polynomials. In \cite{Th}, a thorough description of the space of quadratic laminations was used to produce a
topological model of the Mandelbrot set $\M_2$. In the prequel to this article \cite{paper1,paper2}, we investigated
the space of \emph{symmetric cubic laminations} and constructed the \emph{cubic symmetric comajor lamination} $C_sCL$
that parameterizes them (see Section \ref{s:lami}). In the present
article, we show that $\Dbar / C_sCL$ is a monotone model for the connectedness locus $\M^{sy}_3$ in
the space $\scp$ of \emph{symmetric cubic polynomials} $p_c(z)=z^3-3c^2z$ with \emph{marked critical point} $c$.

\subsection*{Acknowledgments.} The figures in the article have been produced with a modified version of \emph{Mandel}, a software written by Wolf Jung.

\section{Notation and preliminaries}

We assume knowledge of basic facts and concepts of complex dynamics.
We also use standard notation (such as $J(P)$ for the Julia set of a polynomial $P$, etc).

Consider the space $\scp$ of symmetric cubic polynomials $p_c(z)=z^3-3c^2z$ with marked critical point $c$.
Every polynomial (except for $p_0(z)=z^3$) shows in $\scp$ twice, first with $c$ as the marked critical point,
and second with $-c$ as the marked critical point; strictly speaking we are considering the space of pairs
(a polynomial, a critical point). 
Thus, this space is a (branched) two-to-one cover of the moduli space of all odd cubic polynomials, where the moduli
space means the quotient space with respect to complex linear conjugacy. The critical
points of $p_c$ are $c$ and $-c$, the corresponding critical values are $- 2 c^3$ and $2c^3$. The \emph{marked cocritical point} of $p_c$,
i.e. the other preimage of $-2c^3=p_c(c)$, is $-2c$. 
Call sets $A$ and $B$ \emph{mutually symmetric} if $B=-A$. If $A=-A$ we call a set $A$
\emph{symmetric}. Since $p_c$ is odd, the Julia set $J(p_c),$ the filled Julia set $K(p_c)$, and their complements are symmetric.
Observe that $p_c(0)=0$ and $p'_c(z)=p'_c(-z)=3z^2-3c^2$.

Let $\M^{sy}_3$ be the \emph{connectedness locus} of $\scp$, i.e. the set of all $c$ for which the Julia set of $p_c$  is connected.
It is known that the Julia set of $p_c$ is connected if and only if all forward orbits of critical points of $p_c$ are bounded.
Since $p_c$ has 
 mutually symmetric critical orbits, we conclude that 
$c \in \M^{sy}_3$ if and only if the orbit of $c$ or, equivalently of $-2c$ or $2c$, is bounded.

For any $r>0$ set $\D_r=\{z\in\C\,|\,|z|<r\}$, and write $\D$ for $\D_1$.
Let $\uc$ be the unit circle.
For a set $A\subset \C$, let $\overline{A}$ be its closure and $\partial A$ be its boundary.
We use the terms \emph{periodic orbit} and \emph{cycle} interchangeably. External rays to the Julia set of a polynomial $P$ are denoted
$R_\ta(P)$ where $\ta$ is the argument of the ray (if there is no ambiguity we may omit the polynomial from our notation).

Let $X$, $Y$ be topological spaces and $f:X\rightarrow Y$
be continuous. Then $f$ is said to be {\em monotone} if $f^{-1}(y)$ is
connected for each $y \in Y$.
It is known that if $f$ is monotone and $X$ is a continuum then $f^{-1}(Z)$
is connected for every connected $Z\subset f(X)$.

\section{Symmetric cubic laminations} 
\label{s:lami}
Invariant laminations were introduced in \cite{Th}; they play a major role in polynomial dynamics.
The two preceding papers \cite{paper1, paper2} of this series contain an overview,
 which relies upon \cite{Th} and 
 \cite{bmov13}.
The reader is advised to consult Section 2 of \cite{paper1} for more detailed discussion.

If a monic polynomial $P$ has a locally connected Julia set $J(P)$, then $P|_{J(P)}$
 is topologically conjugate to a suitable quotient of the $d$-tupling map $\si_d:\uc\to\uc$
 (where $\si_d(z)=z^d$ if $\uc$ is viewed as the unit circle in $\C$ and as
 $\si_d(\ta)=d\ta$ if $\uc$ is identified with $\R/\Z$).
The quotient is with respect to an equivalence relation $\sim_P$;
 the \emph{leaves} of the corresponding lamination $\lam_P$ are by definition the edges of
 the convex hulls of all $\sim_P$-classes.

A chord of $\cdisk$ with endpoints $a, b$ is denoted by $\ol{ab}$; it is \emph{critical} if
$\si_d(a)=\si_d(b)$ while $a\ne b$.
A lamination is normally denoted by $\lam$ while the union of all its leaves and $\uc$ by $\lam^*$.
A {\em gap} $G$ of a lamination $\lam$ is the closure of a component of
$\D\sm\lam^{\ast}$. A gap $G$ is said to be
\emph{critical} if its image is not a gap or the degree of $\si_d|_{\partial G}$ is greater than 1.
A \emph{critical set} is a critical leaf or a critical gap. For $A\subset \overline{\D}$, denote $A \cap \uc$ by $\V(A)$
and call the elements of $\V(A)$ \emph{vertices} of $G$.
If $G$ is a leaf or a gap of $\lam$, then $G$ coincides with the convex hull of
$\V(G)$. A gap $G$ is called \emph{infinite (finite)} if and only if $\V(G)$ is
infinite (finite).

Let $\lam$ be a lamination. The equivalence relation $\sim_\lam$ induced by
$\lam$ is defined by declaring that $x\sim_\lam y$ if and only if there
exists a finite concatenation of leaves of $\lam$ joining $x$ to $y$.
A lamination $\lam$ is called a
\emph{q-lamination} if the equivalence relation $\sim_\lam$ is laminational
and $\lam$ consists of the edges of the convex hulls of $\sim_\lam$-classes.
Two distinct chords are called
($\si_d$-)\emph{siblings} if they have the same $\si_d$-image.

\subsection{Symmetric laminations}\label{ss:symlam}
Here we follow \cite{paper1} and \cite{paper2}. Let $\tau$ be the rotation of
$\cdisk$ (or of $\uc$) by $180^{\degree}$ around its center $\Oc$. Also, given
a map $f:X\to X$ we call $x\in X$ \emph{preperiodic of preperiod $k>0$} or
\emph{$k$-preperiodic} if $f^k(x)$
is $f$-periodic while $x, f(x), \dots, f^{k-1}(x)$ are not periodic.

\begin{definition}[Symmetric laminations]\label{cs-lam} A $\sigma_3$-invariant
lamination $\lam$ is called a \emph{symmetric (cubic) lamination} if
$\ell\in \lam$ implies $\tau(\ell)\in \lam$.
\end{definition}{}

\begin{definition}[length and majors]
Given a non-diameter chord $\ell$ in $\ol{\D}$, define the arc $h(\ell)$ as the shortest
arc of $\uc$ joining the endpoints of $\ell$. If $\ell'$ and $\ell''$ are chords such that
$h(\ell')\subset h(\ell'')$, then we say that $\ell'$ is \emph{under} $\ell''$.
Define the \emph{length} $|\ell|$ of $\ell$
as the length of $h(\ell)$ divided by $2\pi$ in the case when $\ell$ is not a diameter; if $\ell$ is a diameter,
set $|\ell|=\frac12$. Given a symmetric lamination $\lam$, call a leaf $M$ a \emph{major of $\lam$} if there are no leaves
of $\lam$ closer in length to $\frac13$ than $M$.
\end{definition}

Let $\Gamma:[0, \frac12]\to [0, \frac12]$ be a piecewise linear function with slope $\pm 3$ defined as $\Gamma(x)=3x$
if $0\le x\le \frac16$ and as $\Gamma(x)=|3x-1|$ if $\frac16\le x\le \frac12$. Then $|\si_3(\ell)|=\Gamma(|\ell|)$. Simple analysis
of the dynamics of $\Gamma$ shows that for any leaf $\ell$ an eventual image of $\ell$ has length $0$, or length $\frac12$, or length
which is between $\frac14$ and $\frac{5}{12}$.


Suppose that $\frac14\le |\ell|\le \frac{5}{12}$. Then there exists a sibling chord $\ell'$ of $\ell$ such that
the strip $S$ of $\cdisk$ between $\ell$ and $\ell'$ has two circle arcs on its boundary, and each arc is at most $\frac16$ long.
We   also consider chords $\tau(\ell)$ and $\tau(\ell')$ as well as the $\cdisk$-strip $\tau(S)$ between them.
The union $S\cup \tau(S)$, denoted $\sh(\ell)$, is called the \emph{short strips set} of $\ell$.

A major $M$ of a symmetric lamination $\lam$ can be critical; in this case there is a unique point $x\in \uc$
that is not an endpoint of $M$ but has the same $\si_3$-image as $M$. This point $x$ is called a \emph{comajor (of $\lam$)}. If $M$ is not
critical, then a leaf $M'$ (similar to $\ell'$ from the above) and leaves $\tau(M)$ and $\tau(M')$ are also majors of $\lam$.
Then we set $\sh(\lam)=\sh(M)$ and call this set the \emph{short strips set} of $\lam$.
The third sibling $\bc$ of $M$ that is disjoint from $M\cup M'$, is of length at most $\frac16$. It is called a \emph{comajor (of $\lam$)}.
Similarly we define a \emph{cocritical set} $\co(U)$ of a critical set $U$ as the gap, or the leaf, or the point disjoint from $U$
but with the same image as $U$.

Note that the two majors  $M_{\bc}, M'_{\bc}$ can be easily constructed from a comajor $\bc$
Because of the symmetry, comajors, majors, etc., come in pairs.
A pair of comajors $\bc, \tau(\bc)$ of a symmetric
lamination $\lam$ is called a symmetric \emph{comajor pair}. It is \emph{degenerate} if its elements are points
and \emph{non-degenerate} otherwise. Considering a symmetric lamination $\lam$ we often assume that one
of its majors is marked; we denote this major by $M_\lam$ and denote the corresponding comajor by $\co_\lam$.
Observe that if $\lam_1$ and $\lam_2$ are symmetric laminations such that $\sh(\lam_1)\subset \sh(\lam_2)$
(e.g., if $\lam_2\subset \lam_1$) then the comajors of $\lam_1$ are located under the comajors of $\lam_2$.

\begin{lemma}[\cite{paper1}]\label{l:sh}
Let $\ell$ be a leaf of a symmetric lamination $\lam$ with $|\ell|\ge \frac14$.
If 
$\si_3^n(\ell)\subset \sh(\ell)$ for some $n>0$, then, for the smallest such $n$, the leaf $\si_3^n(\ell)$ 
 non-strictly separates (in $\cdisk$) either $\ell$ from $\ell',$ or $\tau(\ell)$ from $\tau(\ell')$. Thus, either
$\si_3^n(\ell)$ equals one of the leaves $\ell, \ell', \tau(\ell), \tau(\ell'),$ or it is closer to $\frac13$ in length
than $\ell$. In particular, forward images of  majors/comajors of 
 $\lam$ never enter the open circle
arcs on the boundary of the set $\sh(\lam)$.
\end{lemma}

Lemma \ref{l:sh} 
 motivates the next definition.

\begin{definition}[Legal pairs]\label{d:legal} If a symmetric pair
$\{\bc,\tau(\bc)\}$ is either degenerate or satisfies the following conditions:

\begin{enumerate}
    \item[(a)] no two iterated forward images of $\bc, \tau(\bc)$ cross, and

    \item[(b)] no forward image of $\bc$ crosses the interior of $\sh(M_{\bc})$,
\end{enumerate}

\noindent then $\{\bc, \tau(\bc)\}$ is said to be a \emph{legal pair}.

\end{definition}


\begin{lemma}
[\cite{paper1}]
\label{pull-back-lam1}
A legal pair $\{c,\tau(c)\}$ is the comajor pair of the symmetric lamination
$\lam(c)$. A symmetric pair $\{c, \tau(c)\}$ is a comajor pair if and only if it
is legal.
\end{lemma}


A symmetric lamination with an infinite gap such that the map $\si_3$ on it is of degree greater
than 1 is called a \emph{Fatou lamination}.

\begin{lemma}[\cite{paper1}]\label{l:pre1}
A symmetric lamination is Fatou if and only if it
has a preperiodic comajor of preperiod 1.
\end{lemma}


Suppose that $\lam$ is a symmetric q-lamination with
two finite critical gaps each of which is preperiodic of preperiod at least two. Then $\lam$ is
called a \emph{symmetric Misiurewicz lamination}.
%
A symmetric Mi\-siu\-re\-wicz lamination 
 has a well defined pair of comajors.
Suppose that the critical $\sim_\lam$-classes are gaps $G$ and $\tau(G)$ with at least $6$ vertices each.
Then there are two cocritical gaps $H\ne G$ and $\tau(H)\ne \tau(G)$ of $\lam$ such that $\si_3(H)=\si_3(G)$ and
$\si_3(\tau(H))=\si_3(\tau(G))$.
One edge of $H$ and one edge of $\tau(H)$ are the comajors of $\lam$.
Two majors of $\lam$ are edges of $G$ that are siblings of the comajor edge of $H$;
  two other majors of $\lam$ are edges of $\tau(G)$ that are siblings of the comajor edge of $\tau(H)$.
While other edges of $G$ and $\tau(G)$ are not siblings of 
 the comajors, they can be  majors 
of other laminations that are finite tunings of $\lam$.

Indeed, suppose that $\ell$ and $\ell'$ are two sibling edges of $G$ that are not majors.
The convex hull of $\ell\cup \ell'$ is a 4-gon $Q$ with two extra edges $\ol{y}$ and $\ol{q}$ not equal to $\ell$ or $\ell'$.
Construct a new lamination $\lam'$ by inserting $\ol{y}$ and $\ol{q}$ in $G$, pulling them back along the backward orbit of $G$
and then doing the same with $\tau(G)$ and its backward orbit. The majors of $\lam'$ are $\ol{y}$ and $\ol{q}$
and their $\tau$-images. If $\ell''$ is a leaf of $\lam$ which is not an edge of $G$ and is such that $\si_3(\ell'')=\si_3(\ell)$
then $\ell''$ and $\tau(\ell'')$ are the two comajors of $\lam'$.
Repeating this construction for all pairs of sibling edges of $G$ but the majors, we see that
every edge of the cocritical gap $H$ or $\tau(H)$ is a comajor of a certain symmetric lamination which is a
tuning of the original symmetric Misiurewicz lamination $\lam$. Call cocritical sets of Misiurewicz laminations \emph{Misiurewicz}
cocritical sets.
By \cite[Theorem 3.9]{paper1}, symmetric laminations have no wandering gaps.
Therefore, the above is a full description of finite gaps formed by comajors.
The cocritical gaps $H$ and $\tau(H)$
described above will be called \emph{Misiurewicz cocritical gaps}; similarly, if a symmetric Misiurewicz q-lamination
has critical 4-gons (not 6-gons or higher as was assumed above) we call its comajors \emph{Misiurewicz cocritical leaves}.

\begin{theorem}[\cite{paper1, paper2}]\label{t:ccl} 
The set of non-degenerate comajors of symmetric laminations is a q-lamination $\lam$ invariant under
$\tau$ that induces an equivalence relation $\sim_\lam$ on $\uc$.
For any non-degenerate comajor $\bc$
 (i.e., 
 a leaf of $\lam$) one of the following holds.

\begin{enumerate}

\item
It is a not eventually periodic 
two-sided limit leaf in $\lam$.

\item
It is 
an at least $2$-preperiodic 
two-sided limit leaf of $\lam$ (in which case $\bc$ is a Misiurewicz cocritical leaf)
or an edge of a finite gap $H$ of $\lam$ whose
edges are limits of leaves in $\lam$ disjoint from $H$ (in which case $H$ is a Misiurewicz cocritical gap).

\item
It is a $1$-preperiodic 
comajor of a Fatou lamination and
is disjoint from all other leaves of $\lam$; all such comajors $\bc$ are dense in $\lam$
and all 1-preperiodic angles are endpoint of such comajors.
\end{enumerate}
\end{theorem}

\begin{figure}

   \centering
        \includegraphics[height=6.2cm]{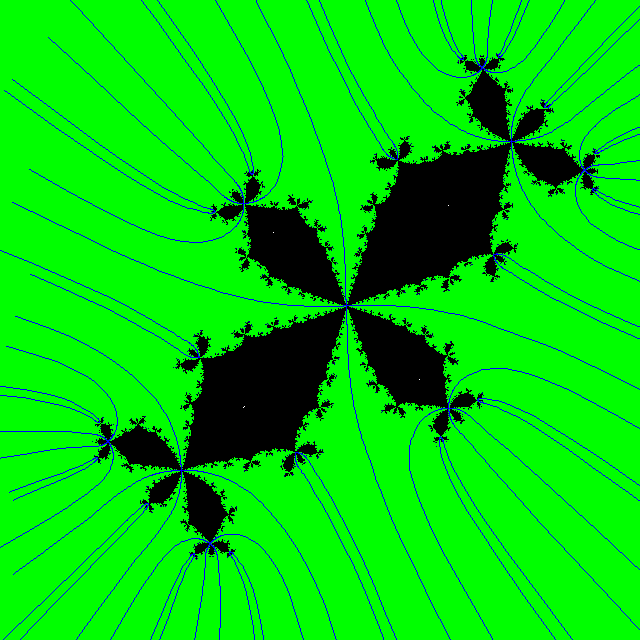}
        \includegraphics[height=6.2cm]{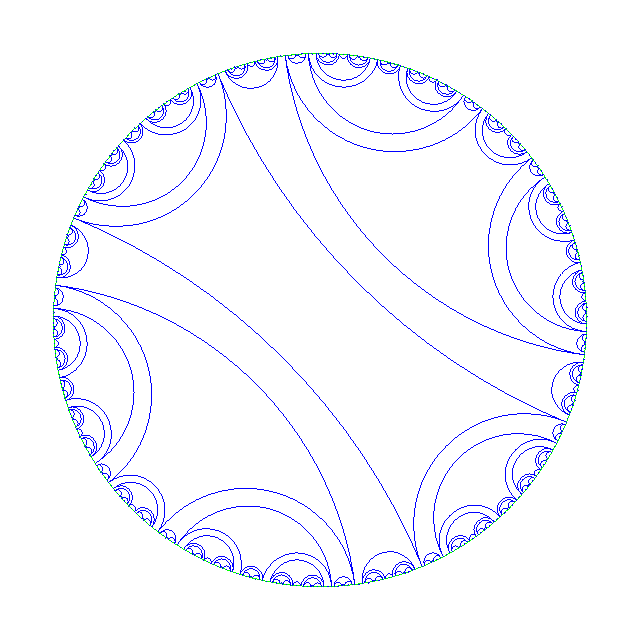}
     \caption{\label{f2} The symmetric cubic lamination with type B comajor $(\frac{5}{48},\frac{7}{48})$ on the right and the Julia set of a corresponding polynomial with external rays on the left.}
\end{figure}

Since comajors are leaves of laminations, their endpoints are either both not preperiodic,
or both preperiodic with the same preperiod and the same period, or both periodic with the same period.
All classes of $\sim_\lam$ from Theorem \ref{t:ccl} are finite. By Theorem \ref{t:ccl}, periodic points of $\uc$
are degenerate comajors.

\begin{definition}\label{d:ccl}
The q-lamination from Theorem \ref{t:ccl} is called the \emph{Cubic Symmetric Comajor Lamination} and is denoted
by $C_sCL$. It induces an equivalence relation denoted $\sim_{sy}$. The $\sim_{sy}$-classes corresponding to
symmetric Misiurewicz laminations are called \emph{Misiurewicz $\sim_{sy}$-classes}. Denote by $\M^{sy}_{3, comb}$
the factor space $\cdisk/\sim_{sy}$. Let $\pr_{comb}:\cdisk\to \M^{sy}_{3, comb}$ be the corresponding quotient map.
\end{definition}

Theorem \ref{t:ccl} verifies the \emph{density of hyperbolicity} conjecture for $C_sCL$.

\begin{lemma}[]\label{l:bd}
Let $\lam$ be a symmetric non-empty Fatou lamination.
Then one of the following holds.

\begin{enumerate}

\item[(B)] 
There is only one cycle $\Ac$ of Fatou gaps of $\lam$. 
It has even period $2m$, and is $\tau$-symmetric.
The periodic majors $M$ and $\si_3^m(M)=\tau(M)$ of $\lam$ 
are edges of critical gaps $U\in\Ac$ and $V=\si_3^m(U)=\tau(U)$. 
Non-periodic majors of $\lam$ 
are siblings of $M$ and $\tau(M)$ and edges of $U$ and $V$, respectively.
The remaining (i.e., not belonging to a major) $2m$-periodic vertices $x, y$ of $U$ are such that $\si_3^m(x)=\tau(y)$
while $\si_3^m(y)=\tau(x)$.

\item[(D)] 
There are exactly
two cycles of Fatou gaps of the same period, 
interchanged by $\tau$.
Critical gaps $U$, $V=\tau(U)$ 
 belong to different cycles.
The periodic majors $M$ and $\tau(M)$ of $\lam$ are edges of $U$ and $V$, respectively. 
Non-periodic majors are siblings of $M$ and $\tau(M)$ and edges of $U$ and $V$, respectively.
\end{enumerate}

In either case 
all edges of infinite gaps are eventually mapped to periodic majors.
The only periodic orbit of edges on the boundary of a Fatou gap of $\lam$ is the orbit of a major
of $\lam$ and it has the same period as the Fatou gap.
\end{lemma}

\begin{figure}

   \centering
        \includegraphics[height=6.2cm]{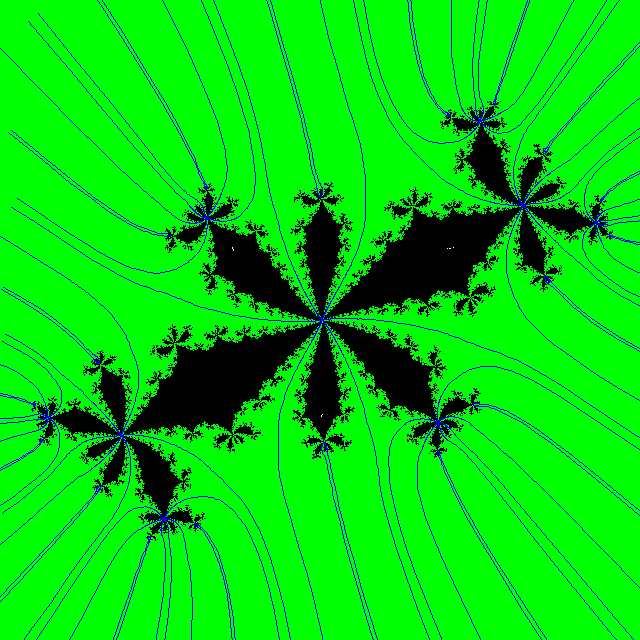}
        \includegraphics[height=6.2cm]{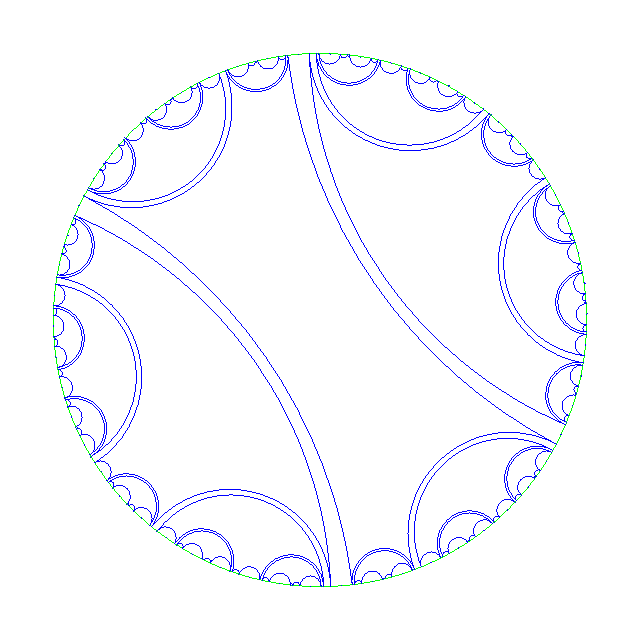}
     \caption{\label{f3} The symmetric cubic lamination with type D comajor $(\frac{7}{78},\frac{4}{39})$ on the right and the Julia set of a corresponding polynomial with external rays on the left.}
\end{figure}

Lemma \ref{l:bd} summarizes the results of \cite{paper1} (compare \cite{paper1}, Lemma 3.8) dealing with symmetric non-empty
(i.e., having some non-degenerate leaves) Fatou laminations. Figures \ref{f2} and \ref{f3} show examples of polynomials of type B and D and the corresponding laminations.

\begin{proof}
All claims of the lemma except for the last one are immediate; observe that the claims concerning the period of the majors
follow from Lemma \ref{l:sh}. To prove the last claim consider an edge $\ell$ of a critical gap $U$ from a cycle $\Tc$ of Fatou gaps.
It is well-known that any edge of $U$ eventually maps to a periodic or a critical edge. 
Since, evidently, $U$ has no critical edges, it suffices to prove that
 the only periodic edge of $U$ is $M$.
Indeed, let $N\ne M$ be a periodic edge of $U$. Then no image of $N$ can be a point (since $N$ is periodic)
or a diameter of $\cdisk$ (since otherwise 
$N$ itself is a diameter invariant under $\si_3$, 
hence $N=M$). Take the closest approach $N'$ in length to $\frac13$ among the images of $N$. By
Lemma \ref{l:sh} an eventual image of $N'$ that is an edge of $U$ must coincide with $M$, a contradiction.
\end{proof}


The terminology below is adopted from \cite{mil93, mil09}, see also \cite{bopt16}.

\begin{definition}\label{d:bd} Symmetric Fatou laminations with properties from Lemma \ref{l:bd}(B)
(respectively, Lemma \ref{l:bd}(D)) are said to be \emph{of type B} (respectively, \emph{of type D}).
\end{definition}

We will also need an immediate corollary of Lemma 6.1 of \cite{paper1}.

\begin{corollary}[\cite{paper1}, Lemma 6.1]\label{c:major2}
Distinct symmetric Fatou laminations have disjoint comajors.
\end{corollary}

Next we consider infinite gaps of $C_sCL$.
One of them 
 plays a special role.
Recall that $\Oc$ is the center of $\D$.
Each comajor is of length at most $\frac16$.
Hence $\Oc$ does not belong to any comajor; it must then lie inside a gap.
The \emph{main gap} $G_{main}$ is by definition the gap of $C_sCL$ that contains $\Oc$ in its interior.

\begin{theorem}\label{t:maingap}
The gap $G_{main}$ is infinite, and $\tau(G_{main})=G_{main}$.
Each edge $\ell$ of $G_{main}$ is a comajor with the same image as the longest edge of a $\si_3$-invariant symmetric
finite rotational gap $H$ and is associated with the symmetric Fatou lamination $\lam_\H$ formed by $H$, Fatou gaps of degree greater than $1$
attached to $H$ and ``rotating'' around $H$, and their pullbacks.
\end{theorem}

\begin{proof}
If $G_{main}$ is finite, then,
 by Theorem \ref{t:ccl}, 
 it is a Misiurewicz cocritical gap of preperiod at least 2 of a symmetric lamination $\lam$. 
This is a contradiction, since 
 then the other cocritical set of $\lam$ will have to contain $\Oc$ and intersect the interior of $G_{main}$. 
Thus, $G_{main}$ is infinite.

Let $\ell$ be an edge of $G_{main}$ and the marked comajor of a symmetric lamination $\lam$.
Since $\Oc\in G_{main}$, then $\ell$ cannot be located under another comajor.
By Theorem  \ref{t:ccl}, the leaf $\ell$ can only be a 1-preperiodic comajor of a Fatou lamination $\lam$.
Let $U$ be the marked critical Fatou gap of $\lam$
with periodic major $M$. If $M$ is the limit of leaves of $\lam$ (necessarily from outside of $U$), then $\ell$ is the limit
 of leaves $\ell_i$ so that $\ell$ is located 
 under $\ell_i$ for any $i$.
 By Lemma 6.6 of \cite{paper1}, this implies that $\ell$ is
the limit of the comajors under which $\ell$ is located, a contradiction.
Hence $M$ is an edge of a periodic gap 
 and is isolated in $\lam$.
Clearly, so is $\tau(M)$.
Let us remove the grand orbits of $M$ and $\tau(M)$ from $\lam$ and consider the resulting family of leaves $\lam'$.
It easily follows (essentially, by definition) that $\lam'$ is again a symmetric lamination. If $\lam'\subsetneqq \lam$ is non-empty, then,
evidently, it has a comajor $\ell'$ such that $\ell$ is under $\ell'$, a contradiction. Thus, $\lam'$ is the empty lamination and, so,
the grand orbits of $M$ and $\tau(M)$ form the entire $\lam$.

Since $\lam$ must have a finite invariant gap $H$, it follows that $\lam$ consists of $H$,
 Fatou gaps attached to $H$ and ``rotating'' around $H$, and their iterated pullbacks.
Observe that $H$ itself must be symmetric under $\tau$.
By Lemma \ref{l:bd}, the lamination $\lam$ can be of type B or D.
By definition, the two
periodic majors of $\lam$ are the closest to criticality edges of $H$ (in this case it is equivalent to being the longest). There are two
comajors of $H$; just like in the case of symmetric polynomials, either of them can be marked, and so in $C_sCL$ the lamination $\lam$ is
reflected twice.
\end{proof}

The following notion will allow us to deal with type B and D laminations in a unified fashion.

\begin{definition}[First (half-)return]
\label{d:eta} Let $\lam$ be a symmetric Fatou lamination and $U$ be a critical gap of $\lam$.
If $\lam$ is of type D and a critical gap $U$ is of period $n$ then set $\eta=\si_3^n|_U$.
If $\lam$ is of type B and a critical gap $U$ is of period $2m$ set $\eta=\tau\circ \si_3^m|_U$.
Thus, $\eta$ is a self-map of $U\cap\uc$; it can also be extended linearly
over the edges of $U$ and, using a barycentric construction, inside $U$.
The map $\eta$ is called the \emph{first (half-)return map} of $U$.
\end{definition}

Strictly speaking, $\eta$ depends on the choice of a symmetric lamination and its gap,
 however, we will not reflect it in writing 
 to lighten the notation.
Lemma \ref{l:eta} is left to the reader. 

\begin{lemma}\label{l:eta}
Let $U$ be a critical gap of a symmetric Fatou lamination $\lam$. Then $\eta$ maps $\partial U$
 onto $\partial U$
in a 2-to-1 fashion and is semiconjugate
to $\si_2$ by the monotone map $\phi$ collapsing edges of $U$ to points.
The fixed point set of $\eta|_{\partial U}$ is the periodic major of $\lam$.
If $\ell\subset \cdisk$ is a chord
 whose endpoints are never mapped 
 to the $\si_2$-fixed point,
 then the $\phi$-preimage of $\ell\cap\uc$ spans a chord in $U$ that has a unique sibling $\ell'\subset \co(U)$.
\end{lemma}

The leaf or point $\ell'$ described in the last claim of Lemma \ref{l:eta} is said to be \emph{induced} by $\ell$.

A parabolic quadratic polynomial from the Main Cardioid has a
$\si_2$-invariant lamination $\lam_2$ called \emph{central}; the major of $\lam_2\ne \0$ is
an edge shared by a finite invariant gap and a critical periodic Fatou gap.

\begin{theorem}\label{t:cardio}
Let $G$ be an infinite gap of $C_sCL$ not containing $\Oc$. Then, for some Fatou
lamination $\lam$, 
 a cocritical Fatou gap $V$ of $\lam$ contains $G$,
and $\partial G$ consists of single points and chords in $V$ corresponding to majors of $\si_2$-invariant central laminations.
In particular, edges of $G$ are $1$-preperiodic while
other vertices of $G$ have infinite orbits.
\end{theorem}

\begin{proof}
Let us consider the \emph{ceiling} of the gap $G$, that is
the unique edge $\ell$ that separates $G$ from $\Oc$. In other words, the gap $G$ is located under $\ell$. 
By Theorem \ref{t:ccl}, the edge $\ell$ is a comajor of a Fatou lamination $\lam$.
Denote by $U$ the critical Fatou gap of $\lam$ with periodic major $M$ such that $\si_3(M)=\si_3(\ell)$.
Then all
 edges of $G$ are associated with symmetric laminations $\lam'$ that tune $\lam$.


Evidently, $\lam'|_U$ is invariant under $\eta$, the map defined before the statement of this theorem.
Since $\eta$ is of degree 2 and is modeled by $\si_2$ (the map $\phi:\partial U\to\uc$ collapses the edges of $U$ and semiconjugates
$\eta$ with $\sigma_2$, as explained above), the edges of $G$ are associated with central quadratic laminations.
This completes the proof.
\end{proof}

\section{Connectedness of $\M^{sy}_3$}

Recall that 
$h^k$ denotes 
the $k$-th iteration of a map $h$.

\begin{lemma}\label{l:90}
The set $\M^{sy}_3$ is invariant under 
the multiplication by $i$.
\end{lemma}

\begin{proof}
We claim that if $f(z)=z^3-az$ and $g(z)=z^3+az$ then
$f^2$ and $g^2$ are conjugate. Indeed, $f$ and $-g$ are conjugate by the map $I:z\mapsto iz$;
hence $f^2$ and $(-g)^2$ are conjugate by $I$.
Since $g$ is an odd function, we have $(-g)^2=g^2$.
Thus, $I$ conjugates $f^2$ and $g^2$.
Since $p_{ic}=z^3+3c^2z$ while $p_c(z)=z^3-3c^2z$,
then $p^2_c$ and $p^2_{ic}$ are conjugate. 
\end{proof}

We now need a construction similar to that for quadratic polynomials; in our description below
we follow the exposition from \cite{Lyubich}.
Take a topological disk $\Delta_c$ around infinity in $\ol{\C}=\C\cup\{\infty\}$ that contains no critical points of $p_c$,
 does not contain $0$, and is such that $p_c(\Delta_c)\subset\Delta_c$ (in particular, $0\notin p_c^n(\Delta_c)$ for $n\ge 0$).
Define B\"ottcher function $B_c(z)=\lim_{n\to\infty} (p_c^n(z))^{1/3^n}$ on $\Delta_c$, where the root is taken so that the corresponding
functions are tangent to the identity at infinity.
The existence of a single valued branch follows from the fact that $\Delta_c$ is simply connected, and that $0$
 does not belong to $p_c^n(\Delta_c)$.
Recall that the Green function $g_c:\C\to\R$ is defined as $g_c(z)=\lim_{n\to\infty} 3^{-n}\log_+|p_c^n(z)|$,
 where $\log_+(t)$ is the maximum of $0$ and $\log t$.
Then $g_c(z)=\log|B_c(z)|$ for all $z\in\Delta_c$.
The \emph{equipotential $E_c(t)\subset \C$} is defined as the level set $\{g_c=t\}$ of the Green function;
 this is a real analytic curve for $t>0$, possibly singular. Note that $B_c$ conjugates $p_c$ and $z^3$ near infinity.

For $c \in \C\sm\M^{sy}_3$, set $\Delta_c$ to be the exterior of the equipotential of $p_c$ passing through $\pm c$
 (the equipotential $E_c(g_c(c))$ has singularities but the exterior of it is a topological disk).
Then the required properties of $\Delta_c$ are fulfilled.
It is easy to see (from the continuous dependence of the B\"ottcher coordinate on parameters)
 that the union $\Uc$ of $\{c\}\times \Delta_c$, where $c$ runs through the complement of $\M^{sy}_3$, is open.
Standard arguments show that $B_c(z)$ is analytic in both $c$ and $z$ on $\Uc$.

If $c\in\C\sm\M^{sy}_3$, then $(c,-2c)$ is always on the boundary of $\Uc$, since the value
 of the Green function of $p_c$ at the cocritical point $-2c$ coincides with those at $\pm c$.
However, the map $(c,z)\mapsto B_c(z)$ extends analytically to a neighborhood of $-2c$.
Moreover, $-2c$ is a regular point of this analytic extension in the sense that $z\mapsto B_c(z)$
 is a conformal injection in a neighborhood of this point. From now on we will assume that $B_c(z)$ is defined
in this neighborhood of $-2c$.

For a fixed $c\notin \M^{sy}_3$ the map $B_c$ is a conformal isomorphism between $\Delta_c$ and the set
$\{z\in \C: |z|>g_c(c)\}=\C\sm \cdisk_{g_c(c)}$.
This defines \emph{initial segments of (dynamical) external rays} $R_c(\ta)$ of $p_c$, i.e. $B_c$-preimages of
the radial rays of argument $\ta$ in $\C\sm \cdisk_{g_c(c)}$. Evidently, these initial segments of external rays of $p_c$ are
orthogonal to all equipotentials $E_c(t), t>g_c(c)$. Moreover, as we mentioned above equipotentials can be defined for any
$t>0$. This allows one to give the following definition: a \emph{smooth external ray} $R$ of $p_c$ is
a smooth unbounded curve that crosses every equipotential orthogonally and terminates in the Julia set of $p_c$.
All but countably many initial segments of external rays defined above extend as smooth external rays. However countably many initial segments
will hit critical points or their eventual preimages (in what follows such points are called \emph{(pre)critical} or eventual) and, therefore,
will not extend as smooth external rays.

Let $\Psi (c)= B_c(-2c)$ be the B\"ottcher coordinate of the marked cocritical point
in the sense of the analytic continuation mentioned above.
Then $\Psi$ is well-defined and holomorphic on $\C\setminus \M^{sy}_3$.
Theorem \ref{t:connect} is analogous to the statement that the Mandelbrot set is connected and
 is proved similarly.

\begin{theorem}\label{t:connect}
The symmetric connectedness locus $\M^{sy}_3$ is compact and connected.
\end{theorem}

\begin{proof}
Let us show that $\Psi$ maps $\C \setminus \M^{sy}_3$ onto $\C \setminus \Dbar$.
First we prove that $\Psi(c) \sim \sqrt[3]2 c$ as $c\to \infty$. Set $z_n=p_c^n(-2c)$ and let $$r_n = z_{n+1}/z_n^3=1-3c^2/z_n^2.$$
If $|c|\ge 2$ and $|z|\ge 2|c|$, then $|p_c(z)/z|\ge 4|c|^2-3|c|^2=|c|^2\ge 4$ and hence, $|p_c(z)|\ge 4|z|$. Thus,
$|z_n|\ge 2\cdot 4^{n-1}|c|^3\ge 2|c|^3$ for $n\ge 1$ (which implies that $\M^{sy}_3\subset \D_2$) and
for any $c\in \D_2$ we have $J(p_c)\subset \D_4$.

We conclude that $|c|\ge 2$ yields 
$$1-\frac{3}{4|c|^4}\le|1-3c^2/z_n^2|=|r_n|\le 1+\frac{3}{4|c|^4}.$$
On the other hand,
$$z_n=r_{n-1}z^3_{n-1}=r_{n-1}r_{n-2}^3z_{n-2}^9=\dots=r_{n-1}r_{n-2}^3\cdots r_2^{3^{n-3}}r_1^{3^{n-2}}(2c^3)^{3^{n-1}}$$
which yields that
$$\sqrt[3^n]{z_n}=(\sqrt[3^n]{r_{n-1}}\sqrt[3^{n-1}]{r_{n-2}}\dots \sqrt[27]{r_2}\sqrt[9]{r_1})(\sqrt[3]{2})c.$$
Using the above bounds on $|r_n|$, the formula for the sum of the geometric series, and the fact that $\Psi(c)=\lim_{n\to \infty} \sqrt[3^n]{z_n}$
we see that
$$\left(\sqrt[6]{1-\frac{3}{4|c|^4}}\right)\sqrt[3]{2}|c|\le \sqrt[3^n]{|z_n|}\le \left(\sqrt[6]{1+\frac{3}{4|c|^4}}\right)\sqrt[3]{2}|c|,$$ 
\noindent that yields the following: as $c\to \infty$, $\Psi(c)=\sqrt[3]{2} c (1+O(1/c^{2/3})).$
Thus, the map $\Psi$ can be continuously extended to $\infty$ with $\Psi(\infty)=\infty$ implying that it can be done holomorphically and
that the local degree of $\Psi$ at $\infty$ is 1. In particular, $\infty$ is in the interior of the range of $\Psi$.

By the above, $\M^{sy}_3\subset \D_2$, in particular, $\M^{sy}_3$ is compact.
Let $M$ be the maximum of the continuous function $(z,c)\mapsto g_c(z)$ on the compact set
 $\ol\D_2 \times\ol \D_4$.
It follows that $|B_c(z)|\le e^M$ for all $c\in\ol\D_2$ and $z\in\ol\D_4$ such that $(c,z)\in\ol\Uc$.

We claim that if $c_n\to \M^{sy}_3$ then $|\Psi(c_n)|\to 1$. Indeed, otherwise
there exists a sequence $c_n\to c_0\in  \M^{sy}_3$
with $|\Psi(c_n)| > e^m>1$. Take $k$ such that ${3^k}m>M$.
Since $|B_c(p_c^k(-2c))|=|(B_c(p_c(-2c))^{3^k}|=|\Psi(c)|^{3^k}$, then $|B_{c_n}(p_{c_n}^k(-2c_n))|>e^M$
 which, by the choice of $M$, implies $|p_{c_n}^k(-2c_n)|>4$ for all sufficiently large $n$
 (indeed, $|c_n|\le 2$ for large $n$).
By continuity, $p_{c_0}^k(-2c_0)\ge 4$; this shows that the cocritical point of $p_{c_0}$ escapes to infinity contradicting the choice of $c_0$.

Since $\infty$ is not on the boundary of the range of $\Psi$ it follows from the above that
$\Psi$ is a proper holomorphic map from $\Cbar \sm \M^{sy}_3$  onto $\Cbar \sm \Dbar$. Hence it is a branched covering with a well-defined degree.
However, the point $\infty$ has exactly one preimage of degree 1; hence $\Psi$ has degree 1 and is actually a conformal isomorphism.
\end{proof}

The function $\Psi(c)$ gives us an analogue of B\"ottcher coordinates for the complement of $\M^{sy}_3$.
In particular, we can define \emph{external parameter rays} (or simply \emph{parameter rays}) as preimages of radial
straight lines under $\Psi$, namely, $\Rc_\theta(r)=\Psi^{-1}(re^{2\pi i\theta})$ with $r\in(1,\infty)$.
The parameter ray $\Rc_\theta$ \emph{lands} at a parameter $w$ if $\lim_{r\to1} \Rc_\theta(r)=w$.
Note that, by definition, the parameter ray $\Rc_\theta$ consists of the parameters $c$ such that the marked cocritical point
$-2c$ of $p_c$ belongs to the dynamical external ray $R_\ta(c)$ of $p_c$. Also note, that the map $\Psi$ is a conformal
isomorphism between $\C\sm \M^{sy}_3$ and $\C\sm \cdisk$ tangential to the  map $z \mapsto \sqrt[3]{2}z$ at infinity.

\section{Hyperbolic components of $\M^{sy}_3$ and their roots}

Hyperbolic components of polynomial parameter spaces play an important role in complex dynamics.
Here we study them for the parameter space of symmetric cubic polynomials.


\subsection{Preliminaries}
We start by recalling basic definitions. 
Let $f:\hat\C\to\hat\C$ be a rational function.
The \emph{multiplier} $\rho(z)$ of a periodic point $z$ of minimal period $n$ under 
 $f$ is defined as the derivative of the first return map to $z$, that is, $\rho(z)=(f^n)'(z)$.
 A periodic
point $z$ is said to be \emph{super-attracting} if $\rho(z)=0$ (which implies that the orbit of $z$ contains a critical point),
\emph{attracting} if  $|\rho(z)|<1$, and \emph{parabolic} if  $\rho(z)$ is a root of unity (which implies
that $(f^{kn})'(z)=1$ for some $k$).


A polynomial $p_c\in \M^{sy}_3$ (and the parameter $c$) is \emph{hyperbolic}/\emph{parabolic} if both of its \emph{finite} critical points are attracted
to \emph{finite} attracting/pa\-ra\-bo\-lic cycles.
To characterize such polynomials we need a lemma.

\begin{lemma}\label{l:0in}
If $U$ is an open 
 topological disk and $-U=U$, then $0\in U$.
If $A$ is a full continuum and $-A=A$, then $0\in A$.
\end{lemma}

\begin{proof}
Set $s(z)=z^2$; then $U=s^{-1}(s(U))$, and $s:U\to s(U)$ is a branched covering.
The claim now follows from the Riemann--Hurwitz formula applied to this covering.
Take a tight symmetric Jordan neighborhood $V$ of $A$ and set $U=s^{-1}(V)$;
 then $0\in U$ by the above.
Since $A$ is the intersection of all such $U$, it follows that $0\in A$.
\end{proof}

We can now describe hyperbolic polynomials $p_c$ more explicitly.

\begin{lemma}\label{l:hyper}
A polynomial $p_c$ is hyperbolic if and only if possesses one of the following:

\begin{enumerate}

\item[(a)]
an invariant symmetric attracting Fatou domain
$p_c$ is 3-to-1 which is the case if and only if $|c|<\sqrt{1/3}$, or

\item[(b)] 
a unique symmetric cycle of attracting Fatou domains
of
 period $2n\ge 2$; 
 there are exactly two mutually symmetric domains
in the cycle containing critical points $\pm c$, 
or

\item[(d)] 
two 
mutually symmetric attracting cycles of Fatou domains.

\end{enumerate}

Thus, if $p_c$ has an attracting cycle, then $p_c$ is hyperbolic.
Also, case (a) is the only case
when a hyperbolic polynomial $p_c$ has a unique bounded periodic Fatou domain.
\end{lemma}

\begin{proof}
If there exists a Fatou domain $U$ on which the map is 3-to-1, then $U$ must be symmetric (otherwise,
$-U$ is another Fatou domain on which the map is 3-to-1, which is impossible).
By Lemma \ref{l:0in}, this implies that
$0\in U$ and $U$ is invariant. Since there must exist a unique fixed point in $U$ and this point must be attracting,
then $0$, being a fixed point, must be attracting.
Since $p'_c(0)=-3c^2$, the corresponding hyperbolic component
of $\M^{sy}_3$ is the round disk of radius $\sqrt{1/3}$ centered at the origin. This corresponds to case (a)
and covers $c=0$ so from now on we assume that $c\ne 0$ and hence $p_c$ has
distinct critical points $c$ and $-c$ with mutually symmetric orbits.

If $c$ (resp., $-c$) is attracted to an attracting cycle, then so is $-c$ (resp., $c$) which implies
that if $p_c$ has an attracting cycle, then $p_c$ is hyperbolic. We can also assume that there are no 3-to-1 Fatou domains.
Now, if $p_c$ has a cycle $A$ of attracting Fatou domains
then by symmetry $-A$ is also a cycle of attracting Fatou domains. Suppose that $A=-A$. By the assumption critical domains in $A$
contain exactly one critical point; since $p_c$ is symmetric, the critical domains in $A$ are mutually symmetric. Moreover,
the fact that $p_c$ is symmetric implies that the first iterate $p_c^n$ that maps either critical domain from $A$ to the other
one is the same for both critical domains which implies that the period of $A$ is $2n$. This corresponds
to case (b). Otherwise $A$ and $-A$ are distinct cycles of Fatou domains which corresponds to case (d).
\end{proof}

Lemma \ref{l:parab} is similar to Lemma \ref{l:hyper}
 and its proof is left to the reader. 

\begin{lemma}\label{l:parab}
Suppose that a polynomial $p_c$ has a parabolic cycle. Then one of the following holds:
\begin{enumerate}


\item[(b)] 
a unique symmetric cycle of parabolic Fatou domains of $p_c$ of period $2n\ge 2$ has
 exactly two 
 mutually symmetric critical domains;

\item[(d)] 
two 
parabolic cycles of Fatou domains of $p_c$ are mutually symmetric.

\end{enumerate}
\end{lemma}


For brevity, a Cremer/Siegel point (cycle) of a polynomial will be referred to as a \emph{CS-point (cycle)}.
Now we show that for CS-cycles the situation 
with symmetric polynomials
is similar to that in Lemma \ref{l:hyper}. First we state a part of
Theorem 4.3 from \cite{BCLOS16} combined with results from \cite{GM93} and \cite{Kiw00}. Define a
\emph{rational cut} as the union of two external rays with rational arguments that land
on the same point called the \emph{vertex} of the cut. If the vertex is a repelling (parabolic) periodic point,
then we call the cut \emph{repelling (parabolic)}.

\begin{theorem}[\cite{BCLOS16,GM93,Kiw00}]\label{t:43}
Let $P$ be a polynomial and $T$ be CS-cycle.
There exists a recurrent critical point $c$ of $P$
and a point $q\in T$ that are not separated by any rational cut of $P$.
Two different objects, each of which is a CS-point, a parabolic domain, or an attracting domain,
 are always separated by a rational cut.
\end{theorem}

Theorem \ref{t:43} is used in the proof of the following lemma.

\begin{lemma}\label{l:cs}
Suppose that a polynomial $p_c$ has a CS-cycle $T$. Then one of the following holds:

\begin{enumerate}

\item[(a)]
the only non-repelling cycle of $p_c$ is $T=\{0\}$,
and neither of the critical points $\pm c$ is separated from $T$ by a rational cut;

\item[(b)]
the only non-repelling cycle of $p_c$ is $T$,
it is a symmetric cycle of period $2n$; 

\item[(d)]
there are exactly 2 non-repelling cycles, 
 namely $T$ and $-T$.
\end{enumerate}
\end{lemma}

\begin{proof}
Assume that $0$ is a CS-point. Then, by Theorem \ref{t:43}, there is a recurrent critical point
$c$ not separated from $0$ by any rational cut, and the same holds for $-c$. This corresponds to case (a) of the
lemma.

 If $0$ is parabolic or attracting then it attracts at least one critical point of $p_c$ and
hence, by symmetry, both of them. In this case $p_c$ has no other non-repelling cycles. So, from now on we assume that $0$ is
repelling.

If $p_c$ has a symmetric CS-cycle $T$ of period $2n$ then, by Theorem \ref{t:43}, it has a point $w\in T$ not separated from a
recurrent critical point, say, $c$, of $p_c$ by a rational cut; hence, $-w\in T$ is not separated from a
recurrent critical point $-c$ of $p_c$ by a rational cut. This, again by Theorem \ref{t:43}, implies that
there are no other non-repelling cycles of $p_c$. This corresponds to case (b) of the lemma.

Finally, let $T$ and $-T$ be distinct mutually symmetric CS-cycles of $p_c$. By Theorem \ref{t:43},
we may assume that $T$ has a point $w$ not separated from a
recurrent critical point, say, $c$ of $p_c$ by a rational cut; hence, $-w\in -T$ is not separated from a
recurrent critical point $-c$ of $p_c$ by a rational cut. 
This implies that
there are no other non-repelling cycles of $p_c$. This corresponds to case (d) of the lemma.
\end{proof}

\subsection{Hyperbolic components and multipliers}

Let $p_{c_0}$ have a periodic point $w$ of period $n$ such that $(p_{c_0}^n)'(w)\ne 1$.
By the implicit function theorem applied to the equation $p_c^n(z)=z$,
there is a holomorphic function $\alpha(c)$ defined on an open Jordan disk around $c_0$ such that
 $\alpha(c_0)=w$, and $\alpha(c)$ is a periodic point of $p_c$ of period $n$.
Also, the multiplier $(p_c^n)'(\al(c))$ is a holomorphic function of $c$.
Hence the set of hyperbolic parameters is an open subset of $\M^{sy}_3$; a connected  component $\Hc$ of this set
is called a \emph{hyperbolic component} of $\M^{sy}_3$. For any $c\in \Hc$ the $\omega$-limit set of the marked critical point is
the unique \emph{marked} attracting cycle $Q_c$; the \emph{period} of $\Hc$ is the period of $Q_c$. Conjecturally,
every connected component of the interior of $\M^{sy}_3$ is hyperbolic.

\begin{theorem}[\cite{Sul85}]
\label{t.components}
Every 
 component of the Fatou set of a rational function is eventually periodic.
In particular, any bounded Fatou domain of a hyperbolic symmetric cubic polynomial
 eventually maps into a cycle of Fatou domains that contains an attracting cycle.
\end{theorem}

Recall that, by Lemma \ref{l:hyper}, the set $\D_{\sqrt{1/3}}$ is a hyperbolic  component of $\M^{sy}_3$.

\begin{definition}\label{d:main}
The set $\D_{\sqrt{1/3}}$ is called the \emph{main hyperbolic component} of $\M^{sy}_3$ and is denoted $\H_{main}$.
\end{definition}

Corollary \ref{c:multiplier} follows from Lemmas \ref{l:hyper}, \ref{l:parab} and \ref{l:cs}. 

\begin{corollary}\label{c:multiplier}
A polynomial $p_c$ has 
one symmetric non-repelling cycle, or two mutually symmetric non-repelling cycles with equal multipliers, or no non-repelling cycles at all.
\end{corollary}

%


Let $\H\ne \H_{main}$ be a hyperbolic component
 and $c\in \H$.
Denote by $F^\pm_c$ the critical Fatou domains of $p_c$
that correspond to the critical Fatou gaps $U^\pm_\H$, respectively, of $\lam_\H$
(we always assume that the marked critical point $c$ belongs to $F^+_c$).
Let $M_\H$ and $M'_\H$ be the majors of $\lam_\H$ that are edges of $U^+_\H$;
then $\si_3(M_\H)=\si_3(M'_\H)$ by Lemma \ref{l:bd}, and we always assume that $M_\H$ is periodic.
Let $\co^+_\H$ be the marked comajor (i.e., $\si_3(\co^+_\H)=\si_3(M_\H)$)
 and let $\co^-_\H$ be the other comajor of $\lam_\H$.
Similar objects can be defined
for any hyperbolic or parabolic parameter $c$ yielding such notation as $U^\pm_c,$ $M_c,$ $M'_c$.

Let a hyperbolic component $\H$ be given.
All $c\in\H$ satisfy the same option $(a)$, $(b)$ or $(d)$ from Lemma \ref{l:hyper}.
According to these three cases, $\H$ is said to be of \emph{type} A, B or D, respectively.
Also, given a polynomial $p_c$ with parabolic (or attracting) periodic point $z$,
let $m_r(c)$ be the minimal number such that $p_c^{m_r(c)}$ fixes dynamical external rays landing on $z$
(or the point $z$ itself if it is attracting); note that it does not depend on the choice of a particular ray and is called the \emph{ray period} of $p_c$. The number
$(p_c^{m_r(z)})'(z)$ is called the \emph{ray multiplier} (of $p_c$) and is denoted $r\hspace{-2pt}\rho(c)$. 
Notice that the period of a parabolic point $z$ may be strictly smaller
than $m_r(c)$ so that $m_r(c)$ is a multiple of the period. 

A similar concept can be defined for a hyperbolic component $\H$. Namely, let $r\hspace{-2pt}\rho_\H:\H\to \D$ be defined as
$r\hspace{-2pt}\rho_\H(c)=r\hspace{-2pt}\rho(c)$. As we show in Theorem \ref{t:preperiodiccomponents}, $r\hspace{-2pt}\rho_\H$ can be extended over $\ol{\H}$. For
a parabolic parameter $c\in \partial \H$, this extended function $r\hspace{-2pt}\rho_\H$ may not be equal to $r\hspace{-2pt}\rho(c)$. 
To emphasize that we use the subscript $_{\H}$
in the notation, and call $r\hspace{-2pt}\rho_\H$ the \emph{ray multiplier based on $\H$}.
This difference is not present for parameters $c\in \H$, but does show for parameters $c\in \partial \H$.

\begin{theorem}
\label{t:preperiodiccomponents}
For a type D hyperbolic component $\H$ of $\M^{sy}_3$
the map $r\hspace{-2pt}\rho_\H$ can be extended onto 
 $\ol\H$ so that
$r\hspace{-2pt}\rho_\H:\ol{\H} \to \ol{\D}$ is a homeomorphism conformal on $\H$.
\end{theorem}

\begin{proof}
The result essentially follows from Theorem C of \cite{IK12}, however, we need to explain
 how the terminology of Inou--Kiwi relates to ours.
Let $\lam$ be a cubic invariant q-lamination with at least one cycle of Fatou gaps.
With $\lam$, one associates a \emph{reduced mapping schema} $T(\lam)$.
Instead of giving a general definition of mapping schemata, we give an explicit description
 of $T(\lam)$ in the case when $\lam$ is symmetric of type D, that is, when $\lam$
 has two distinct cycles of Fatou gaps.
In this case, $T(\lam)$ can be represented as the graph with two vertices and two (directed) edges
 that are loops based at both vertices.
Every edge of $T(\lam)$ is in general equipped with a positive integer called the \emph{degree};
 in our specific case, the degrees of both loops are equal to 2.
Intuitively, the arrows of $T(\lam)$ represent the first return maps to the critical Fatou gaps of $\lam$.
The space $\Cc(T(\lam))$ in our specific case consists of all pairs $(q_0,q_1)$ of
 monic centered quadratic polynomials $q_0$, $q_1$
 with connected Julia sets --- informally, the two loops of $T(\lam)$ are replaced with $q_0$ and $q_1$.

By definition, the space $\Rf(\lam)$ consists of all monic cubic polynomials $f$ such that
\begin{itemize}
  \item the filled Julia set $K(f)$ is connected;
  \item for any (pre)periodic leaf of $\lam$ with endpoints $\alpha$, $\beta$, the corresponding external rays
   $R_\alpha(f)$ and $R_\beta(f)$ land on the same (pre)periodic point of $K(f)$;
  \item let $U_0$, $U_1$ be the critical Fatou gaps of $f$;
  the corresponding subcontinua $K_0$ and $K_1$ of $K(f)$ are polynomial-like filled Julia sets
  of certain polynomial-like restrictions of $f^n$, where $n$ is the period of $U_0$ and $U_1$.
\end{itemize}
Here, one needs to explain in which sense $K_0$ corresponds to $U_0$, and similarly with $K_1$ and $U_1$.
For each $\ell\in\lam$, let $\alpha$ and $\beta$ be the endpoints of $\ell$, and write $\Gamma_\ell$
 for the cut formed by the external rays $R_\alpha(f)$, $R_\beta(f)$, and their common landing point.
Then $K_0$ corresponding to $U_0$ means that $K_0$ lies on the same side of $\Gamma_\ell$ as $U_0$ relative to $\ell$,
 for every $\ell\in\lam$.
By the Douady--Hubbard straightening theorem, $f^n$ is hybrid equivalent to a
 unique monic quadratic polynomial $q_i$ near $K_i$, where $i=0$, $1$.
The \emph{Inou--Kiwi straightening map} (abbreviated as \emph{IK-straightening map}) $\chi_{\lam}:\Rf(\lam)\to\Cc(T(\lam))$ takes $f$ to $(q_0,q_1)$.
(The fact that $q_i$ are quadratic yields some simplification: in higher degree cases one needs additional normalization called
 \emph{internal angles assignment} in order to make $q_i$ unique).

Theorem C of \cite{IK12} can now be formulated as follows.
\emph{Denote by $\mathrm{Hyp}(\Cc(T(\lam)))$ the set of hyperbolic maps contained in $\Cc(T(\lam))$
 (in our specific case, $(q_0,q_1)\in\Cc(T(\lam))$ being hyperbolic means both $q_0$ and $q_1$ are hyperbolic).
Then $\chi_\lam(\Rf(\lam))\supset\mathrm{Hyp}(\Cc(T(\lam)))$, the inverse image
 of $\mathrm{Hyp}(\Cc(T(\lam)))$ under $\chi_\lam$ is an open set, and the restriction of $\chi_\lam$
 onto this open set is biholomorphic.}
Now let $\H$ be a given type D hyperbolic component of $\M^{sy}_3$; it lies in some
 hyperbolic component $\hat\H$ of the connectedness locus of all monic centered cubic polynomials.
All polynomials from $\hat\H$ have the same lamination, say, $\lam$.
By Theorem C of \cite{IK12}, the restriction of $\chi_\lam$ to $\hat\H$ is a biholomorphic isomorphism between
 $\hat\H$ and the product of the interior $\mathrm{Ca}$ of the main cardioid with itself.
The image of $\H$ under $\chi_\lam$ is then the diagonal in $\mathrm{Ca}\times\mathrm{Ca}$,
 and the restriction of $\chi_\lam$ is a biholomorphic map between $\H$ and this diagonal
 (the latter is isomorphic to $\D$ under the multiplier map).
It follows that $r\hspace{-2pt}\rho_\H:\H\to\D$ is a conformal isomorphism.
Clearly, it extends to a homeomorphism $r\hspace{-2pt}\rho_\H:\ol\H\to\ol\D$.
\end{proof}


For any hyperbolic component $\H$, write $r\hspace{-2pt}\rho_\H$ for the restriction of the function $r\hspace{-2pt}\rho$ to $\H$.
Below, we discuss the boundary extension of $r\hspace{-2pt}\rho_\H$, which may not coincide with $r\hspace{-2pt}\rho$
 at points of $\partial\H$ where the latter function is defined.
Let a polynomial $p_c$ have a parabolic or attracting cycle $X_c$ of type B.
Then the period of $X_c$ is an even number $2n$, and $p_c^n(x)=-x$ for every $x \in X_c$.
The number $-(p_c^n)'(x)$ does not depend on the point $x\in X_c$, is denoted $r\hspace{-2pt}\tro(c)$, and is called the \emph{ray half-multiplier} (of $p_c$).
Note that $r\hspace{-2pt}\tro(c)$ can be interpreted as the multiplier of the fixed point $x$ of the map $-p_c^n$.
As before, the ray half-multiplier depends only on the parameter and is defined as long as $p_c$ is of type B.
If $c\in \H$, where $\H$ is of type B, then we write $r\hspace{-2pt}\tro_\H$ for the restriction of
 the function $r\hspace{-2pt}\tro_\H(c)$ to $\H$.
The \emph{ray (half-)multiplier} of $c$ is defined as either $r\hspace{-2pt}\rho(c)$ or $r\hspace{-2pt}\tro(c)$ depending
 on whether $c$ is type D or type B.

\begin{theorem}
\label{t:preperiodiccomponents2}
Let $\H$ be a hyperbolic component of $\M^{sy}_3$ of type B. 
The map $r\hspace{-2pt}\tro_\H$ can be extended over $\ol\H$ so that $r\hspace{-2pt}\tro_\H:\ol{\H}\to \ol{\D}$ is a homeomorphism
which is conformal on $\H$ while
the ray multiplier $r\hspace{-2pt}\rho_\H=r\hspace{-2pt}\tro_\H^2:\ol{\H}\to \ol{\D}$ is a double-covering.
\end{theorem}

\begin{proof}
Similarly to Theorem \ref{t:preperiodiccomponents},
Theorem \ref{t:preperiodiccomponents2} also follows from Theorem C of \cite{IK12}.
Let $\hat\H$ be the hyperbolic component in the full space of monic centered cubic polynomials containing $\H$.
All polynomials from $\hat\H$ have the same lamination $\lam$ so that the
 corresponding reduced mapping schema $T(\lam)$ has two vertices and two directed edges connecting the
 two vertices in opposite directions; each edge has degree 2.
The space $\Cc(T(\lam))$ consists of pairs $(q_0,q_1)$ of monic centered quadratic polynomials
 such that the Julia set of $q_1\circ q_0$ is connected.
The straightening map of Inou--Kiwi provides a biholomorphic isomorphism between $\hat\H$
 and the \emph{principal hyperbolic component} in $\Cc(T(\lam))$ consisting of all $(q_0,q_1)$
 such that $q_1\circ q_0$ has an attracting fixed point, and $J(q_1\circ q_0)$ is a Jordan curve.
Clearly, symmetric cubic polynomials are mapped to pairs $(q_0,q_1)$, for which $q_0=q_1$.
The corresponding slice of the principal hyperbolic component is biholomorphic to $\D$,
 and the corresponding conformal isomorphism is given by the multiplier of the attracting fixed point of $q_0=q_1$.
\end{proof}

Let us now define the center and the root of a hyperbolic component.

\begin{definition}[Center and root]\label{d:center}
Let $\H$ be a hyperbolic component of type B or D.
The \emph{center} of $\H$ is the point $c\in \H$ such that $p_c$ has
superattracting cycle(s). 
The \emph{root} of $\H$ is the point $r_\H\in\partial\H$ such that
$r\hspace{-2pt}\rho_\H(r_\H)=1$ (if $\H$ is of type D) or
$r\hspace{-2pt}\tro_\H(r_\H)=1$ (if $\H$ is of type B).
\end{definition}

Lemma \ref{l:para1} justifies the above definition of the roots.

\begin{lemma}\label{l:para1}
Suppose that 
 $z$ is a parabolic point of $p_c$ 
 of ray period $n$. Then there is a major $\ol{\ta\ta'}$ of $\lam_c$ such that both
$R_\ta(c)$ and $R_{\ta'}(c)$ land on $z$.
Moreover, the 
(half-)multiplier of $c$ equals $1$.
%
%
%
%
\end{lemma}

\begin{proof}
Let $U$ be a period $n$ parabolic domain attached to $z$,
 and $F$ be the Fatou gap of $\lam_c$ corresponding to $U$.
Then $F$ has a unique periodic major $M=\ol{\ta\ta'}$. If $U$ is of type D then
the only $p_c^n$-fixed point in $\partial U$ is $z$ and the only $\si_3^n$-fixed points in $\partial F$ form the major $M$. Hence
$R_\ta(c)$ and $R_{\ta'}(c)$ land on $z$. If $U$ is of type B then $n=2m$ and there are three $p_c^n$-fixed points in $\partial U$ associated with
$M$ and two vertices $x, y$ of $F$. Let $x', y'\in \partial U$ be points associated with vertices $x, y$ of $F$. We claim that $x',y'$ are not parabolic.
Suppose that $x'$ is parabolic.
Then $p_c^m(y')=-x'$ by part (B) of Lemma \ref{l:bd}, which implies that $y'$
is also parabolic, a contradiction. Hence only $M$ can be associated with $z$ as desired.
Let us now prove 
 that the (half-)multiplier of $c$ equals 1.

(B) Let $p_c$ be of type B, $z\in \partial F^+_c$ be parabolic, and $F^+_c$ be of period $2m$; then $-p_c^m(z)=z$, and, by Lemma \ref{l:eta},
 the map $-p_c^m$ fixes the two external 
 rays of $p_c$ landing on $z$. Hence $(-p_c^m)'(z)=1$, which implies that $r\hspace{-2pt}\tro(c)=1$.

(D) Similar to (B) (details are left to the reader).
\end{proof}

The main component $\H_{main}$ of $\M^{sy}_3$ (for which 0 is attracting) is a round disk of radius $\sqrt{3}/3$. For $|c|<\sqrt{3}/3$
the Julia set $J(p_c)$ is a Jordan curve and $p_c|_{J_c}$ is conjugate to $\si_3$. In other words, the map is neither of type B nor of type D,
and neither Theorem~\ref{t:preperiodiccomponents} nor Theorem~\ref{t:preperiodiccomponents2}
applies. Each polynomial $p_c$ with $|c|=\sqrt{3}/3$ has a neutral fixed point at $0$ of multiplier $-3c^2$.
If $c_{1, 2}=\pm \frac{i}{\sqrt{3}}$ then
the corresponding polynomials $p_{c_{1, 2}}=z^3+z$ have multiplier 1 at fixed point $0$.
It is easy to see that the lamination associated with $z^3+z$ has the
 leaf $\ol{0\,\frac12}$, which is the horizontal diameter of the unit circle,
 and two invariant Fatou gaps located, respectively, above and below $\ol{0\,\frac12}$;
 the presence of these invariant gaps completely determines the lamination.
In particular, the laminations at these two parameters are different from the empty lamination which corresponds to any polynomial in $\H_{main}$. 
We will not consider these points (or any other points) as
roots of $\H_{main}$ so that $\H_{main}$ has the center (at the origin) but does not have a root.

\section{Parabolic polynomials}

The following stability property will be used repeatedly.
We state it only for symmetric cubic polynomials.

\begin{theorem}[Lemma B.1 of \cite{GM93}]
\label{t:gm}
Let $w$ be a repelling periodic point of $p_{c_0}$ such that an external
dynamical ray $R_{\theta_0}(c_0)$ with rational argument $\theta_0\in\R/\Z$
lands on $w$. Let $\ta$ be an angle such that an eventual $\si_3$-image of $\ta$ belongs to the $\si_3$-orbit
of $\ta_0$ and the rays from the (finite) $p_{c_0}$-orbit of $R_\ta(c_0)$ are all smooth
and do not land on critical points. Then, for $c$ sufficiently close to $c_0$,
the dynamical external rays from the 
$p_c$-orbit of $R_{\theta}(c)$ are smooth,
move continuously with $c$, and land on preperiodic or repelling periodic points of $p_c$ close to the landing
points of the dynamical external rays of $p_{c_0}$ with the same argument.
\end{theorem}

Recall that the lamination $\lam_c$ is defined for all $c$ such that $K(p_c)$ is locally connected.
In particular, if $\H$ is a hyperbolic component, then for any $c\in \H$ the lamination
$\lam_c$ exists and is independent of the choice of $c$.
Denote it by $\lam_\H$; clearly, $\lam_\H$ is a symmetric Fatou lamination.

\begin{definition}[Geometrically finite and sub-hyperbolic \cite{hai00}] A polynomial $f$ is \emph{geometrically finite}
if all its critical points are preperiodic points or attracted by parabolic or attracting
cycles. If, moreover, $f$ has no parabolic points, then it is said to be \emph{sub-hyperbolic}.
\end{definition}

Theorem \ref{t:parattr} based on \cite{hai00} and additional arguments.

\begin{theorem}[\cite{hai00}]\label{t:parattr} The following holds.

\begin{enumerate}

\item A parabolic symmetric polynomial $p_c$  is accessible from a hyperbolic component
$\H$ of $\M^{sy}_3$ such that $\lam_\H=\lam_c$.

\item If $\H'$ is a hyperbolic component and $c'\in \partial \H'$ is such that $\lam_{c'}=\lam_{\H'}$, then
$c'=r_{\H'}$ and $\co^+_{c'}=\co^+_\H$.

\end{enumerate}

\end{theorem}

\begin{proof}
(1) Let $p_{c}$ be the given parabolic polynomial. It is of type B or D.
By the main theorem of \cite{hai00}, there exists a path $f_t$ of monic centered cubic polynomials such that
$f_0=p_c,$ $f_t$ is sub-hyperbolic for $t>0$, and $f_t|_{J(f_t)}$ is topologically conjugate to $p_c$ on its Julia set.
Since $p_c$ is of type B or D, the polynomial $f_t$ is hyperbolic for any $t>0$ and, hence, there exists a hyperbolic component
$\widehat\H$ in the space of all monic centered cubic polynomials that contains the path $\{f_t\}_{t>0}$.

We now use the description of $\widehat\H$ given in Theorems \ref{t:preperiodiccomponents} and \ref{t:preperiodiccomponents2}.
In both cases (type D or type B), the straightening map ($\rho$ or $\tro$, resp.)  yields a biholomorphic parametrization of $\widehat\H$
by pairs $(q_0,q_1)$ of monic centered quadratic polynomials.
A path in $\widehat\H$ converging to $p_{c}$ can be represented as $(q_0(t),q_1(t))$, where $t\in (0,1]$,
and the latter path converges as $t\to 0$ to some $(q,q)$ since both multipliers or half-multipliers have the same limit.
It follows that the path in $\widehat\H$ represented by $(q_0(t),q_0(t))$ also converges to $p_c$.
On the other hand, this new path consists of symmetric cubic polynomials.
Thus $p_c$ is accessible from $\H=\widehat\H\cap\scp$, as desired.

(2) The claim follows from the assumption that $\lam_{c'}=\lam_{\H'}$ and Lemma \ref{l:para1}.
\end{proof}

The following is a combinatorial description of hyperbolic components of $\M^{sy}_3$.

\begin{theorem}
\label{t:center1}
The map taking a hyperbolic component $\H$ (other than $\H_{main}$) of $\M^{sy}_3$ to the corresponding marked comajor $\co^+_\H$
 is a bijection between the hyperbolic components of $\M^{sy}_3$ (with the exception of $\H_{main}$) and $1$-preperiodic comajors of $C_sCL$.
Two distinct marked parabolic polynomials must have distinct marked comajors.
\end{theorem}

\begin{proof}
Bijection follows from results of A. Poirier \cite{poi92}, which
 in turn extend earlier work of Bielefeld--Fisher--Hubbard \cite{BFH92}.
Namely, injectivity is
Theorem 1.1 
and surjectivity is Theorem 1.3. 
More precisely, the map 
defined in \cite{poi92} sends $\H$ to the corresponding \emph{critical portrait}.
On the other hand, there is a 
bijection between the marked comajors of $C_sCL$
 and symmetric cubic Fatou critical portraits of period $>1$.

It remains to prove the last claim.
Suppose now that we have two parabolic marked polynomials $p_c$ and $p_{c'}$ associated with the same marked comajor $\ell$.
By Theorem, \ref{t:parattr}, the parameter $c$ is the root of a
unique hyperbolic domain $\H$ such that $\co^+_\H=\ell$, and, similarly, $p_{c'}$ is the root of a
unique hyperbolic domain $\H'$ such that $\co^+_{\H'}=\ell$.
By the above, $\H=\H'$.
Hence $c=c'$ is the root of $\H$, and $p_c=p_{c'}$ is the same marked polynomial.
\end{proof}

We will now study the place of a parabolic parameter $c$ in the parameter space.
More precisely, given a hyperbolic component $\H$ we show that $\lam_\H=\lam_{r_\H}$.
We also consider parameter rays $\Rc_\ta$ with 1-preperiodic argument $\ta$, show that they
land on a parabolic parameter $c$, and relate $\ta$ and the marked comajor associated with $p_c$.

We will interchangeably use notation $3\ta$ and $\si_3(\ta)$ for any $\ta\in \uc$. Recall also that
$\tau$ denotes the rotation of $\cdisk$ or $\uc$ by $180^{\degree}$.

\begin{lemma}\label{l:nogmajor}
If $\lam'\subset \lam$ are two symmetric laminations such that $\lam'\ne \0$ is Fatou
then $\lam$ cannot contain a finite periodic gap $G$ located inside a Fatou gap $U'$ of $\lam'$
with an edge which is a major of $\lam'$.
\end{lemma}

\begin{proof}
Suppose by way of contradiction that the gap $G$ as described in the lemma exists and $M=\ol{xy}$ is a major of $\lam'$ and an edge of $G$.
Consider the first (half-)return map $\eta$ (see Definition \ref{d:eta}) of $U'$ with respect to $\lam'$.
Then $\eta|_{\partial U'}$ is two-to-one. If $G$ exists, it is $\eta$-invariant, which is impossible.
Indeed, $\eta|_{\partial U'}$ is two-to-one and semiconjugate to $\si_2$
by the map collapsing the edges of $U'$ to points; by Lemma \ref{l:bd} all edges of $U'$ are
preimages of its major.
Hence the existence of $G$ implies the existence of an invariant leaf or gap of $\si_2$ with a $\si_2$-fixed endpoint,
 which is absurd.
\end{proof}


\begin{lemma}\label{l:landing1}
Every parameter ray $\Rc_\ta$ at a 1-preperiodic angle $\theta$ lands on a parabolic parameter $c_0$
on the boundary of $\M^{sy}_3$. Moreover, one of the rays $R_{\ta\pm \frac13}(c_0)$
lands on a parabolic periodic point of $p_{c_0}$.
\end{lemma}

\begin{proof}
We claim that every parameter ray $\Rc_\ta$ at a 1-preperiodic angle $\theta$ lands on a parabolic parameter $c_0\in \partial
\M^{sy}_3$, and the dynamical ray $R_{3\ta}(c_0)$ lands on a $p_{c_0}$-parabolic point.
Indeed, let $c_0\in\partial \M^{sy}_3$ be in the accumulation set of $\Rc_\theta$. Recall that for $c\in \Rc_\ta$,
the dynamical external ray $R_\ta(c)$ passes through the cocritical point $-2c$ escaping to infinity
under the action of $p_c$. Thus, the dynamical external ray $R_{3\ta}(c)$ has a periodic argument $3\ta$
and, therefore, is not smooth because an eventual $\si_3$-image of $3\ta$ is the argument of an external ray that
terminates at the critical point $c$. In particular, $R_{3\ta}(c)$ is not smooth.

On the other hand, the dynamical external ray $R_{3\ta}(c_0)$ of $p_{c_0}$ is periodic and
lands on a repelling or parabolic periodic point $z_0$ of $p_{c_0}$. Hence $R_\theta(c_0)$
lands on a non-periodic point $z_1$ such that $p_{c_0}(z_1)=z_0$.
Since $c_0\in \M^{sy}_3$, all dynamical external rays of $p_{c_0}$ are smooth.
If $z_0$ is a repelling point of $p_{c_0}$, then, by Theorem \ref{t:gm},
the dynamical external ray $R_{3\ta}(c)$ is smooth and
lands on a repelling periodic point for all $c$ close
to $c_0$. However, by the previous paragraph, if $c\in \Rc_\theta$ then
$R_{3\ta}(c)$ is not smooth. This shows that $z_0$ is a parabolic point of $p_{c_0}$.

Since the dynamical ray $R_{3\theta}(c_0)$ lands on $z_0$, the period of 
 $z_0$ divides the $\si_3$-period $N$ of $3\theta$.
Moreover, the multiplier of $z_0$ with respect to $p_{c_0}^N$ is equal to 1.
That leaves finitely many candidates for an accumulation parameter $c_0$.
Indeed, parabolic parameters $c_0$, for which there exists an $N$-periodic point of multiplier 1,
form an algebraic subset of $\scp$; this algebraic set is either finite or the entire plane, the latter option being clearly nonsensical.
As  the accumulation set of a ray is connected, it consists of exactly one such parabolic parameter.
It follows also that one of the dynamical external rays $R_{\ta\pm \frac13}(c_0)$ lands on a parabolic point of $p_{c_0}$.
\end{proof}

Recall that if $J(f)$ is locally connected, then there is a well-defined $\si_d$-invariant lamination $\lam_f$.

\begin{definition}[Repelling leaves]
  \label{d:repleaf}
Let $f$ be a degree $d>1$ monic polynomial.
For such $f$, define \emph{repelling} (or \emph{$f$-repelling}) leaves of $\lam_f$ as leaves corresponding
 to pairs of rays landing on the same repelling periodic point of $f$ or an iterated preimage thereof.
Similarly, we can talk of \emph{parabolic} (or $f$-\emph{parabolic}) leaves of $\lam_f$.
\end{definition}

Repelling leaves and related laminations were used in \cite{botw23} in the proof of continuity of the constructed there
monotone (except for one point) model of the entire cubic connectedness locus.

\begin{lemma}\label{l:repel}
Suppose that a sequence of monic degree $d>1$ polynomials $f_n$ converges to a polynomial $f$
 (necessarily monic of degree $d$) as $n\to\infty$.
Then the following holds.

\begin{enumerate}

\item Let the dynamical external rays of $f_n$ with periodic arguments $\ta_1,$ $\dots,$ $\ta_m$ land on the same point
 (neither $m$ nor the angles $\ta_i$ depend on $n$).
If the landing points of $R_{\ta_1}(f)$, $\dots$, $R_{\ta_m}(f)$ are repelling, then
 they coincide.

\item If 
$J(f)$ and all $J(f_n)$ are locally connected, $\lam_{f_n}=\lam$ for some lamination $\lam$,
 and no critical point of
 $f$ is mapped to a repelling periodic point, then all repelling leaves of $\lam_f$ belong to $\lam$.

\end{enumerate}

\end{lemma}

\begin{proof}
The lemma follows from 
Theorem \ref{t:gm}.
\end{proof}

We are going to apply Lemma \ref{l:repel} to the situation where $f$ is on the boundary of a hyperbolic domain $\H$ in some
parameter space of polynomials and $f_n\in \H$ for each $n$.

\begin{lemma}\label{l:forcing}
Suppose that a parameter $c\in\scp\sm\M_3^{sy}$ is such that the rays $R_c(\ta\pm\frac{1}{3})$ hit a critical point.
Consider any $\si_3$-periodic polygon $T$ whose iterated forward $\si_3$-images are disjoint from
 the critical leaves connecting $\ta+\frac 13$ with $\ta-\frac 13$
 and $\ta-\frac 16$ with $\ta+\frac 16$.
Then all dynamical external rays of arguments that are vertices of $T$ land on the same point.
\end{lemma}

This statement is not new; it can be deduced, e.g., from a much more
 general Theorem 7.1 of \cite{dMM}, which gives a model for the landing pattern
 of all external rays of $p_c$ (see also \cite[Theorem 5.4]{Cal} for a restatement of this result in the
 language of invariant laminations).
For completeness, we give a sketch in our specific situation.

\begin{proof}[Sketch of a proof]
  Consider all external rays in the dynamical plane of $p_c$ whose arguments correspond to vertices of $T$.
All these rays land, perhaps at different points.
Let $\hat T$ be the union of these rays with some continuum so that $\hat T$ is connected and disjoint from
 the forward orbits of the critical points of $p_c$. Denote by $p$ the period of $T$.
For every $n=0,1,\dots$, let $\hat T_n$ be the $p_c^{pn}$-pullback of $\hat T$ that contains the same collection of external rays as $\hat T$ and
 the $\si_3^{pn}$-pullback of the connecting continuum along the orbit of $T$.

Since $p_c$ is hyperbolic, the sequence $\hat T_n$ converges to a connected set, comprising the original periodic rays and a continuum containing their landing points. However, all points of the limiting continuum must be inside of the Julia set, so this continuum has to be a singleton since the Julia set is totally disconnected.
\end{proof}

We can now prove a key theorem describing the limit transition of laminations in our situation.

\begin{theorem}\label{t:samelam}
If $c$ is parabolic then the following holds.

\begin{enumerate}

\item If $c=r_\H$ for a hyperbolic component $\H\ne \H_{main}$ then $\lam_c=\lam_\H$
and $\co^+_{r_\H}=\co^+_\H$.

\item If $c$ is the landing point of a parameter ray $\Rc_\ta$ and $\ta$ is 1-pre\-pe\-riodic
then $\ta$ is an endpoint of the marked comajor $\ol{\al\be}$ of $\lam_c$.

\end{enumerate}

\end{theorem}

\begin{proof}
Consider cases corresponding to claims (1) and (2) from the theorem.

(1) 
By Lemma \ref{l:repel}, all repelling leaves of $\lam_c$ are contained in $\lam_\H$.
On the other hand, by Lemma \ref{l:para1} the only parabolic leaves of $\lam_c$ are the majors and their
iterated pullbacks.
If these parabolic leaves are not contained in $\lam_\H$, then the leaf or finite gap $G$  whose vertices correspond to the external rays of $p_c$ landing on the  parabolic point is in the interior of the
 marked Fatou
 gap $U^+_\H$ of $\lam_\H$.
Note that the first (half)return map $\eta$ of $U^+_\H$ (see Definition \ref{d:eta})
 gives a nontrivial combinatorial rotation of $G$.
This contradicts the definition of the root point, namely, that
 $r\hspace{-2pt}\rho_\H(c)=1$ for $\H$ of type D or $r\hspace{-2pt}\tro_\H(c)=1$ for $\H$ of type B.
Thus, $\lam_c\subset\lam_\H$.

Suppose that $\lam_c\subsetneqq \lam_\H$, and $\lam_\H$ has
a finite leaf or gap $G$ in the interior of $U^+_c$ where $U^+_c$ is the marked critical Fatou gap of $\lam_c$.
By Lemma \ref{l:nogmajor}, the set $G$ is disjoint from the majors of $\lam_c$.
By Lemma \ref{l:para1}, there is a repelling periodic point corresponding to $G$,
 a contradiction with Lemma \ref{l:repel}.

(2) Let $V^+_c$ be the closure of the maximal open subset of $\D$ containing $U^+_c$ and disjoint
from all repelling leaves of $\lam_c$. Since there are no fixed return triangles of $\si_3$ by
Lemma 4.4 of \cite{paper1}, the boundary of $V^+_c$ consists
of a Cantor subset of $\uc$ and countably many pairwise disjoint leaves of $\lam_c$.
Clearly, $U^+_c\subset V^+_c$. We claim that if $V^+_c\ne U^+_c$ then there exists
a gap $G\subset V^+_c$ such that all images of $U^+_c$ inside $V^+_c$ share an edge
with $G$. Indeed, suppose otherwise. Then it follows from \cite{Kiw00} that there are repelling
cutpoints of $J(p_c)$ such that the convex hulls of the arguments of dynamical external rays
of $p_c$ landing on them separate $V^+_c$, a contradiction with the definition of $V^+_c$.
This proves the existence of $G$ with the listed  properties. Moreover, it follows that the
edges of $G$ are the periodic leaves of $\lam_c$ inside $V^+_c$.

By Lemma \ref{l:repel}, the angles $\ta \pm \frac13$ are vertices of $V^+_c$.
By Lemma \ref{l:landing1}, one of the rays $R_{\ta \pm \frac13}(c)$ (say, it is
$R_{\ta+1/3}(c)$) lands on a parabolic periodic point $z_c$ of $p_c$. By way of contradiction,
suppose that $\ta\notin \{\al, \be\}$. Set $\ol{AB}$ to be the marked major of $\lam_c$; by assumption $\theta+\frac{1}{3} \notin \{A,B\}$. Then $U^+_c\ne V^+_c$ and we can consider the gap $G$
defined in the previous paragraph. Since at least two rays land on each parabolic point of $p_c$,
the dynamical external rays whose arguments are vertices of $G$ land on $z_c$. Thus, the three angles
$A, B$ and $\ta+\frac{1}{3}$ are vertices of $G$ and the dynamical external
rays of arguments $A, B$ and $\ta+\frac{1}{3}$ land on $z_c$.
Note that the major edge of $G$ coincides with the marked major $\ol{AB}$ of $\lam_c$.

Similarly to the construction of the degree two first (half-)return map $\eta$ (see Definition \ref{d:eta})
we can define a degree two map $\eta':V^+_c\to V^+_c$ semiconjugate to $\si_2$ by collapsing edges of $V^+_c$ to points.
Under this semiconjugacy $G$ maps to a finite $\si_2$-invariant gap (or leaf) $G'$. Moreover,
the critical leaf $\ol{\ta-\frac{1}{3} ,\ta+\frac{1}{3}}$ projects to a critical leaf $\ell$ in $\cdisk$ with a $\si_2$-periodic endpoint
which is a vertex of $G'$ but not an endpoint of the Thurston major of $G'$ (see \cite{Th}). Therefore there
exists a finite $\si_2$-invariant gap $T'$ disjoint from $\ell$. Then the lifting of $T'$ gives a finite $\eta$-invariant
gap $T\subset V^+_c$. Note that $G$ and $T$ have disjoint sets of vertices.

Now consider a point $c_t$ on the parameter ray $\Rc_\ta$,
 where the parameter $t$ corresponds to the value of the Green function for $\M_3^{sy}$ at $c_t$
 so that $c_t$ converges to $c$ as $t\to 0$.
Then the rays $R_{\ta\pm\frac{1}{3}}(c_t)$ both hit the marked critical point $c_t$.
By Lemma \ref{l:forcing}, the dynamical external rays of $p_{c_t}$ whose arguments
 correspond to the vertices of $T$ land on the same repelling periodic point $w_t$.
The point $w_t$ has a well-defined limit $w_0$ as $t\to 0$.
On the other hand, for every angle $\gamma$ corresponding to a vertex of $T$,
 the ray $R_\gamma(c)$ lands on a repelling periodic point of $p_c$.
 By Theorem \ref{t:gm} and by
 Lemma \ref{l:repel} this point is close to $w_t$ for small $t$, hence it must coincide with $w_0$.
Thus, $w_0$ is repelling for $p_c$, and the gap or leaf $T$ corresponding to $w_0$ must belong to $\lam_c$,
 a contradiction.
\end{proof}

Observe that, by Lemma \ref{l:landing1}, there is a dense set of 1-preperiodic angles such that the corresponding
parameter rays land on parabolic parameters in $\M^{sy}_3$.

\begin{theorem}\label{t:paraland}
Let $\ol{\al\be}$ be the marked comajor of a symmetric Fatou lamination $\lam$.
Then there exists a unique hyperbolic component $\H\ne \H_{main}$ such that
the parameter rays $\Rc_\al,$ $\Rc_\be$ land on the root point $r_\H$. Moreover,
let $c\in \partial \H, c\ne r_\H$ be a parabolic point. Then $\lam_c\supsetneqq \lam_\H$, and
the marked comajor of $\lam_c$ is located under the marked comajor of $\lam_\H$.
\end{theorem}

\begin{proof}
Consider the marked comajor $\ol{\al\be}$ of a symmetric Fatou lamination $\lam$.
Then $\al$ and $\be$ are 1-preperiodic.
By Lemma \ref{l:landing1}, the parameter ray $\Rc_\al$ lands on a parabolic parameter $c$.
The angle $\al$ is an endpoint of the marked comajor $\co^+_c$ of the lamination $\lam_c$,
 by Theorem \ref{t:samelam}.
Since distinct 1-preperiodic comajors are
disjoint, $\co^+_c=\ol{\al\be}$.
By Theorem \ref{t:parattr}, there exists a hyperbolic component $\H$ such that $c=r_\H$ is the
root of $\H$.
By Theorem \ref{t:samelam}, the lamination $\lam_\H$ coincides with $\lam_c$.
By Theorem \ref{t:center1}, the set $\H$ is a unique hyperbolic component such that
$\co^+_\H=\ol{\al\be}$. It follows that $\Rc_\be$ lands on $c$, too.

To prove the last claim of the theorem, let
$c\in \partial \H, c\ne r_\H$, be a parabolic parameter.
Since $c\ne r_\H$, the repelling periodic points of polynomials $p_{c^*}\in \H$
associated with the marked major $M_\H$ of $\lam_\H$ converge to a repelling periodic point of $p_c$
 of the same period as $c^*\to c$.
Hence all the edges of the gaps $U^\pm_\H$ of $\lam_\H=\lam_{c^*}$ remain edges of $\lam_c$.
This implies that the critical gaps of $\lam_c$ are contained in $U^\pm_\H$,
and, hence, $\lam_c\supset \lam_\H$, and the comajors of $\lam_c$ are located under those of $\lam_\H$
(this yields the claim of the theorem
about the marked comajors).
The hyperbolic component with the same marked comajor as $\lam_c$
 cannot coincide with $\H$, since $c$ is not a root point of $\H$.
Therefore,
$\lam_c\ne \lam_\H$, which completes the proof.
\end{proof}

Call the parameter rays from Theorem \ref{t:paraland} \emph{characteristic rays of a hyperbolic component $\H\ne \H_{main}$}.
Recall that by an arc $(a, b)\subset \uc$ we always mean the \emph{positively oriented circle arc}
with endpoints $a, b\in \uc$.

\begin{lemma} \label{c:only2} The characteristic rays are the only two strictly preperiodic
rays that accumulate on a parabolic parameter $c$.
\end{lemma}

\begin{proof}
By Theorem \ref{t:parattr}, the point $c$ is the landing point of the parameter rays $\Rc(\al)$ and $\Rc(\be)$
where $\ol{\al\be}$  is a 1-preperiodic comajor.
By Theorem \ref{t:paraland}, all 1-preperiodic comajors give rise to true cuts in the parameter plane.
Hence, if a comajor separates $\ol{\al\be}$ from an angle $\ga$ in $\cdisk$, then $\Rc_\ga$ cannot accumulate on $c$.
Since, by
Theorem \ref{t:ccl}, the 1-preperiodic comajors are dense in $C_sCL$ and disjoint from all other comajors,  it follows that the only way
a parameter ray $\Rc_\ga$ can accumulate on $c$ is when $\ga$ is a vertex of an infinite gap $G$ with an edge $\ol{\al\be}$.
However, by Theorem \ref{t:cardio},  the fact that $\ga$ is preperiodic implies that $\ga$ is actually 1-preperiodic.
Again by Theorem \ref{t:parattr}, this implies that $\Rc_\ga$ cannot land on $c$, as desired.
\end{proof}


\begin{theorem}\label{t:mainpoly}
Each parabolic parameter $c\in \partial \H_{main}$ is associated to a comajor $\ol{\al\be}$
which is an edge of $G_{main}$ and, accordingly, to a point of $\cdisk/C_sCL$. The parameter rays $\Rc_\al$ and
$\Rc_\be$ land on $c$. For every hyperbolic domain $\H\ne \H_{main}$ the corresponding marked comajor $\co_\H^+=\ol{\ta_1\ta_2}$
is associated to the parameter rays $\Rc_{\ta_1}$ and $\Rc_{\ta_2}$ that land on the root
$r_\H$ of $\H$ and separate $\H$ from $\H_{main}$.
\end{theorem}

\begin{proof}
Let $c\in \partial \H_{main}$ be a parabolic parameter. Then $0$ is a parabolic point associated to
a finite invariant gap $T$ whose vertices are the arguments of dynamical external rays of $p_c$ landing on $0$.
It follows that there is a unique symmetric lamination $\lam$ associated with $T$ which has the gap $T$,
Fatou gaps of degree greater than 1 attached to $T$ and ``rotating'' around $T$, and pullbacks of all these gaps
(this fully describes $\lam$). Let $\ol{\al\be}$ be the marked comajor of $\lam$ associated with the marked
cocritical point $-2c$ of $p_c$;
 by Theorem \ref{t:parattr}, the parameter rays $\Rc_\al$ and $\Rc_\be$ land on $c$.
Thus, all parabolic parameters $c\in \partial \H_{main}$ correspond to edges of the main gap $G_{main}$ of $C_sCL$ and in the end
map to the corresponding points of $\Dbar / C_sCL$ that belong to the main domain $D_{main}$ of $\Dbar / C_sCL$. The rest
of the theorem easily follows.
\end{proof}

\section{Misiurewicz parameters}

A number $c\in \C$ is a \emph{Misiurewicz parameter} if the $p_c$-orbits of critical values are
strictly preperiodic. If $c$ is a Misiurewicz parameter, then $K_c$ is connected (i.e., $c\in \M^{sy}_3$), all $p_c$-periodic points
are repelling (see Theorem \ref{t:43}), and $J(p_c)$ is a dendrite (in particular, $J(p_c)$ is locally connected).
Recall that a cubic symmetric lamination $\lam$ is called a \emph{Misiurewicz} lamination if its critical
sets are strictly preperiodic. Such laminations and their comajors are discussed in detail right after
Lemma \ref{l:pre1}.
By Theorem \ref{t:ccl}(2), Misiurewicz cocritical leaves (gaps) are leaves (gaps) of $C_sCL$
approached from all sides by 1-preperiodic comajors.
Since for each polynomial from $\scp$ a critical point is marked, then 
each Misiurewicz lamination is considered twice, with either cocritical set marked. Finally,
recall that the lamination $C_sCL$ defines a laminational equivalence relation $\sim_{sy}$.
For brevity by ``$\sim_{sy}$-class'' we will mean 
 an equivalence class of $\sim_{sy}$.

\begin{lemma}\label{l:dendr}
Let $p_c$ be a cubic symmetric polynomial with a dendritic Julia set, and
 $T$ be the marked critical set of $\lam_c$.
If $c_n\in\scp\sm\M_3^{sy}$ converge to $c$ as $n\to\infty$, then
 the arguments of the external rays of $p_{c_n}$ hitting the marked critical point $c_n$
converge to vertices of $T$. 
\end{lemma}

\begin{proof}
Every leaf of $\lam_c$ that is not an edge of a gap is approximated by (pre)periodic leaves from both
sides, and every edge of a gap $G$ in $\lam_c$ is approximated by preperiodic leaves from outside of $G$.
Choose a neighborhood $W$ of $T$ in $\D$ whose boundary is formed by (pre)periodic leaves chosen close to the
edges of $T$, and appropriate circle arcs.
By Theorem \ref{t:gm}, there is a neighborhood $\Wc$ of $c$ in $\scp$ such that for $c^*\in \Wc$,
the leaves forming the boundary of $W$ in $\D$ are associated with cuts formed by the dynamical external rays of $p_{c^*}$
(it does not matter whether $J(p_{c^*})$ is connected or not).
The part of the dynamical plane of $p_{c^*}$ bounded by all these cuts contains the point $c^*$.
This implies the desired.
\end{proof}

\begin{theorem}\label{t:mis}
Let $\lam$ be a Misiurewicz lamination. 
Parameter rays whose arguments are vertices of the marked cocritical set of $\lam$
land on a Misiurewicz parameter $\hat c$ such that $\lam_{\hat c}=\lam$.
\end{theorem}

\begin{proof}
Let $\ta$ be a $k$-preperiodic angle with $k>1$.
Choose $c_0\in\partial \M^{sy}_3$ in the accumulation set of $\Rc_\theta$.
Then the periodic dynamical external ray $R_{3^k\theta}(c_0)$ lands on a periodic point $z_0$.
By Lemma \ref{c:only2}, the point $z_0$ is repelling.
For a parameter $c$, consider
the union $R(c)$ of the rays from the forward orbit of the closure of $R_\ta(c)$;
it consists of finitely many rays and their landing points.
By Lemma \ref{l:dendr},
there are no critical points among their landing points provided that $c$ is close to $c_0$.

By Theorem \ref{t:gm}, for some neighborhood $\Wc$ of $c_0$ in $\scp$, 
 the set $R(c)$ depends continuously on $c\in\Wc$,
consists only of smooth rays and their landing points,
and contains no critical points of $p_c$. In particular, $R_\ta(c_0)$ lands on the cocritical point $-2c_0$.
Thus, $p_{c_0}$ is a symmetric Misiurewicz polynomial, and $\lam_{c_0}$ is a symmetric Misiurewicz lamination.
Since there are countably many symmetric Misiurewicz polynomials,
 and the accumulation set of $\Rc_\ta$ is either a non-degenerate
continuum (hence uncountable) or a point, it follows that $\Rc_\ta$ lands on $c_0$ and that $\ta$ is
 a vertex of a cocritical set of $\lam_{c_0}$.

By Theorem \ref{t:ccl}, all preperiodic angles of preperiod $>1$ are partitioned into
 vertex sets of various Misiurewicz cocritical sets.
In particular, for a preperiod $>1$ angle $\ta$ 
there exists a unique symmetric lamination $\lam$ such that $\ta$ is a vertex of a cocritical set of $\lam$.
Taking into account the fact that each polynomial is counted in $\scp$ twice (depending on the choice of the marked critical point),
and choosing marked cocritical sets accordingly, we see that if $\lam$ is a Misiurewicz lamination
 whose marked cocritical set has vertices $\ta_1$, $\dots$, $\ta_m$, then the parameter rays
$\Rc_{\ta_1},$ $\dots,$ $\Rc_{\ta_m}$ land on a Misiurewicz parameter $\hat c$ such that $\lam_{\hat c}=\lam$.
\end{proof}

\section{The Structure of $\M^{sy}_3$}

Recall that $\M^{sy}_{3, comb}$ is the factor space of $\cdisk$ under $\sim_{sy}$ and
$\pr_{comb}:\cdisk\to \M^{sy}_{3, comb}$ is the corresponding quotient map (see Definition \ref{d:ccl}).
Given a continuum $K\subset \C$, define the \emph{topological hull of $K$} as the complement to the unbounded
complementary component of $K$. Recall also, that a \emph{monotone} map is a map whose point-preimages (fibers) are connected.

\begin{theorem}\label{t:main}
There exists a monotone continuous surjective map $\pi: \M^{sy}_3\to \M^{sy}_{3, comb}$.
If $\M^{sy}_3$ is locally connected, then $\pi$ is a homeomorphism.
\end{theorem}

\begin{proof}
We use a general construction from \cite{bco11} (for basic continuum theory facts
see, e.g., \cite{nad92}). Associate to each angle $\ta$ the \emph{impression} of the parameter ray $\Rc_\ta$
(abusing the terminology we will call them \emph{impressions of angles $\ta$}).
Declare two angles equivalent if their impressions are non-disjoint, extend this equivalence relation on $\uc$
by transitivity, and then consider the  closure of  the thus constructed equivalence relation. The resulting equivalence relation
$\approx$ on $\uc$ is laminational. Extend it onto the entire $\cdisk$ and then on $\C$ by declaring all points in convex hulls of
$\approx$-classes equivalent, and otherwise ---
 i.e., for all points $z$ not in convex hulls of $\approx$-classes --- declaring
that $z$ is only equivalent to $z$. Keeping the notation $\approx$ for the new equivalence relation (now defined on $\C$)
we obtain a monotone quotient map $\Upsilon:\C\to \C/\approx$, where $\C/\approx$ is homeomorphic to $\C$
by Moore's Theorem.

Define another map $\Psi:\C\to \C$ as follows.
For every $\approx$-class $\mathbf g$, take the union of all impressions
 of angles from that class; by Theorem 1 (see also Lemma 13) of \cite{bco11} 
 this is a continuum, say, $K_{\mathbf g}$.
Define $\Psi:\C\to\C$ as a monotone surjective map whose fibers are
 the topological hulls $\thl(K_{\mathbf g})$ and singletons.
In other words, topological hulls of the sets $K_{\mathbf g}$ are collapsed to points and the rest of the plane is
 mapped forward homeomorphically.
By Theorem 1 of \cite{bco11}, 
 the map $\Psi$ 
 is well defined, monotone and onto;
 moreover, we will show that there is a homeomorphic identification between the range of $\Psi$ and the range of $\Upsilon$,
 under which $\Psi(\M^{sy}_3)=\Upsilon(\cdisk)$, and so
$\Upsilon(\cdisk)$ can be viewed as a monotone model of $\M^{sy}_3$.
This model is optimal in the sense that any other monotone image of $\M^{sy}_3$ is a monotone image of $\Upsilon(\cdisk)$. 
In particular, if $\M^{sy}_3$ 
is locally connected, 
 then $\M^{sy}_3$ and $\Upsilon(\cdisk)$ are homeomorphic.

Thus, it suffices to show that $\Upsilon(\cdisk)$ and $\M^{sy}_{3, comb}$ are the same, i.e., that
the equivalence relation $\approx$ on $\uc$ coincides with $\sim_{sy}$.
By Theorem \ref{t:ccl}, if $\mathbf g$ is a $\sim_{sy}$-class whose convex hull does not intersect an infinite gap of $C_sCL$,
 then $\mathbf{g}$ is in fact a $\approx$-class.
It remains to consider $\sim_{sy}$-classes $\hb$ whose convex hulls intersect infinite gaps of $C_sCL$. In what follows we use notation and terminology
from right before Theorem \ref{t:cardio}. In particular, recall that $\Oc$ is the center of $\cdisk$.

Let $\hb$ be a $\sim_{sy}$-class whose convex hull $\ch(\hb)$ intersects $\partial G$,
 where $G$ is an infinite gap of $C_sCL$.
Suppose that $G$ does not contain $\Oc$.
By Theorems \ref{t:ccl} and \ref{t:cardio}, there is a Fatou lamination
$\lam$ with critical Fatou gap $U$ and 
 the (half-)return map $\eta:\partial U\to \partial U$
semiconjugate to $\si_2:\uc\to \uc$ by a map $\phi$ collapsing all edges of $U$.
The gap $G$ lies in $U$, and the $\phi$-images of the edges of $G$ are
 the majors of laminations from the Main Cardioid of $\si_2$.

Let $\ell$ be the edge of $G$ separating all other edges of $G$ from $\Oc$.
By the properties of the Main Cardioid,
 another infinite gap of $C_sCL$ is attached to $G$ at any edge of $G$ different from $\ell$
 while the points of $\partial G$ that are not endpoints of edges of $G$ are themselves $C_sCL$-classes.
By Theorem \ref{t:ccl}, the leaf $\ell$ may be an edge of another infinite gap of $C_sCL$
or it may be non-isolated in $C_sCL$. 

By Theorem \ref{t:center1}, there is a unique hyperbolic component $\Hc$ such that $\co^+_\H=\ell$.
A dense subset of $\partial \Hc$ consists of parabolic parameters $c$
 with two parameter rays landing on $c\in \partial \Hc$ corresponding to the endpoints of an edge of $G$,
 by Theorem \ref{t:paraland}.
Properties of impressions and the fact that $\partial \H$ is a Jordan curve imply that for all edges of $G$ other than $\ell$
the desired claim (that the corresponding $\uc/\approx$-class and $\sim_{sy}$-class coincide) holds.

The situation with $\ell$ is different.
First, it may be that $\ell$ is an edge of another infinite gap of $C_sCL$ not containing $\Oc$.
Then the above arguments show that the $\uc/\approx$-class associated with $\ell$
consists only of the endpoints of $\ell$,
as desired.
Second, $\ell$ may be the limit of other edges of $C_sCL$ converging to $\ell$ from outside of $G$.
In the latter case, let $\Gamma_\ell$ be the parameter cut corresponding to $\ell$.
By Theorem \ref{t:ccl}, 
 the cut $\Gamma_\ell$ is approximated by 1-preperiodic parabolic cuts.
These cuts separate $\Gamma_\ell$ from impressions of other parameter rays in
 the complementary component of $\Gamma_\ell$ not containing $\H$,
 and imply that in this case, too, the $\uc/\approx$-class and the corresponding $\sim_{sy}$-class coincide.

The remaining case is when $\ell$ is the shared edge of $G$ and the gap $G_{main}$ of $C_sCL$ corresponding to $\Hc_{main}$.
Recall that $\partial \Hc_{main}$ is the circle of radius $\frac{\sqrt{3}}{3}$ centered at the origin.
For $c\in \partial \Hc_{main}$,
the multiplier at the neutral point $0$ is $-3c^2$. Hence there is a dense in $\partial \Hc_{main}$ set of parabolic parameters
$c$ associated to the comajors like $\ell$ above. The arguments from above show that for all edges of $G_{main}$ as well as for
all points of $\partial G_{main}$ that are not endpoints of an edge of $G_{main}$ the same conclusion holds: the $\uc/\approx$-classes and the $\sim_{sy}$-classes
are the same.
\end{proof}

\end{document}